\documentclass[12pt,twoside,reqno]{amsart}
\usepackage[colorlinks=true,citecolor=blue]{hyperref}
\usepackage{mathptmx, amsmath, amssymb, amsfonts, amsthm, mathptmx, enumerate, color}

\usepackage{graphicx}
\usepackage{epstopdf}
\usepackage{caption}
\usepackage{subcaption}
\usepackage{pdflscape}

\usepackage{graphicx}
\usepackage{graphics}
\usepackage{lineno}
\usepackage{physics}
\usepackage{mathrsfs}
\usepackage{algorithm}
\usepackage{array}
\usepackage{multirow}
\usepackage{booktabs}
\usepackage[table,xcdraw]{xcolor}
\usepackage{csquotes}
\usepackage{hyphenat}
\usepackage{mathtools}
\usepackage{environ}
\usepackage{multicol}
\usepackage{adjustbox}
\usepackage{lineno}
\usepackage{mathtools}
\usepackage{commath}
\usepackage{algpseudocode}
\usepackage[utf8]{inputenc}
\usepackage{float}
\usepackage{lipsum}
\usepackage[section]{placeins}
\usepackage{algpseudocode}
\usepackage{libertine}
\usepackage{mathrsfs} 

\makeatletter
\newcommand{\thickhline}{%
    \noalign {\ifnum 0=`}\fi \hrule height 1pt
    \futurelet \reserved@a \@xhline
}

\newtheorem{theorem}{Theorem}[section]

\newtheorem{lemma}{Lemma}[section]
\newtheorem{proposition}{Proposition}[section]
\theoremstyle{definition}
\newtheorem{definition}{Definition}[section]
\newtheorem{example}{Example}[section]

\numberwithin{equation}{section}

\begin{document}
\setcounter{page}{1}
\title[A Quasi-Newton Method for Uncertain Multiobjective Optimization Problems]
{A Quasi-Newton Method to Solve Uncertain Multiobjective Optimization Problems with Uncertainty Set of Finite Cardinality}
\author[K. Gupta, D. Ghosh, C. Tammer, X. Zhao, J. C. Yao]{Krishna Gupta$^1$, Debdas Ghosh$^{1}$, Christiane Tammer$^2$, Xiaopeng Zhao$^{3,*}$, Jen-Chih Yao$^{4,5}$}
\maketitle
\vspace*{-0.6cm}
\begin{center}
{\footnotesize {\it
$^1$Department of Mathematical Sciences, Indian Institute of Technology (BHU), Varanasi, UP, 221005, India\\
$^2$Department of Mathematics, Martin-Luther-University Halle-Wittenberg, 06099, Halle (Saale), Germany\\
$^3$School of Mathematical Sciences, Tiangong University, Tianjin, 300387, China \\ 
$^4$Center for General Education, China Medical University, Taichung, 40402, Taiwan\\ 
$^5$Academy of Romanian Scientists, Bucharest, 50044, Romania
}}\end{center}

\renewcommand{\thefootnote}{}
\footnotetext{ $^*$Corresponding author.
\par
E-mail addresses: krishnagupta.rs.mat23@itbhu.ac.in (K. Gupta), debdas.mat@iitbhu.ac.in (D. Ghosh), \\ christiane.tammer@mathematik.uni-halle.de (C. Tammer), 
zhaoxiaopeng.2007@163.com (X. Zhao), \\ 
yaojc@mail.cmu.edu.tw (J. C. Yao) \\ 
Received 15 August 2024; Accepted 1 February 2025. 
\par
}

\vskip 4mm {\small\noindent {\bf Abstract.}
In this article, we derive an iterative scheme through a quasi-Newton technique to capture robust weakly efficient points of uncertain multiobjective optimization problems under the upper set less relation. It is assumed that the set of uncertainty scenarios of the problems being analyzed is of finite cardinality. We also assume that corresponding to each given uncertain scenario from the uncertainty set, the objective function of the problem is twice continuously differentiable. In the proposed iterative scheme, at any iterate, by applying the \emph{partition set} concept from set-valued optimization, we formulate an iterate-wise class of vector optimization problems to determine a descent direction. To evaluate this descent direction at the current iterate, we employ one iteration of the quasi-Newton scheme for vector optimization on the formulated class of vector optimization problems. As this class of vector optimization problems differs iterate-wise, the proposed quasi-Newton scheme is not a straight extension of the quasi-Newton method for vector optimization problems. Under commonly used assumptions, any limit point of a sequence generated by the proposed quasi-Newton technique is found to be a robust weakly efficient point of the problem. We analyze the well-definedness and global convergence of the proposed iterative scheme based on a regularity assumption on stationary points. Under the uniform continuity of the Hessian approximation function, we demonstrate a local superlinear convergence of the method. Finally, numerical examples are presented to demonstrate the effectiveness of the proposed method. \\

}

\section{Introduction}
Uncertain multiobjective optimization problems (UMOPs) encompass a category of optimization problems where multiple uncertain objectives of a conflicting nature must be addressed. In many practical optimization problems, the objectives are inherently uncertain due to incomplete information, imperfect data, unpredictable factors, or uncertain predictions \cite{1}. Analysis of the effects of uncertain scenarios in optimization theory serves many crucial purposes, such as enhancing the robustness of solutions, improving decision-making under variability, trade-off analysis and managing risk in practical applications.

The area of uncertain (robust) optimization is growing quickly. For a detailed discussion on fundamental concepts and applications of robust optimization, we refer to exploring the research contributions of Ben-Tal et al. \cite{2}. Ehrgott et al. \cite{3}, Kuroiwa and Lee \cite{4} introduced the concept of minimax robustness for uncertain multiobjective optimization problems, replacing the objective of the uncertain problem with a multiobjective function that accounts for the worst-case outcomes of each objective component. This concept of robustness involves converting an uncertain multiobjective optimization problem into a deterministic one, known as the robust counterpart. Various types of robust counterparts (see \cite{36} and Section $15.4$ in \cite{23}) can be derived for uncertain multiobjective optimization problems. Approaches like the minimax strategy and the objective-wise worst-case robust counterpart \cite{3} are used to address uncertainty by creating a deterministic formulation. These deterministic problems can be addressed using scalarization techniques, such as the popular normal-boundary intersection or the weighted sum method. For more information on these techniques, refer to Ehrgott \cite{5}. In \cite{6,7}, various definitions of solution concepts for UMOPs were proposed and reported based on several set less relations \cite{8}. For additional insights, the references \cite{9,10,11,12} offer detailed information on these developments, and for a new robust counterpart through aggregation operator, see recently developed ``generalized ordered weighted aggregation robustness'' in \cite{13,14}.

Regarding numerical methods for solving UMOPs, the authors in \cite{15,16} effectively tackled UMOPs with a finite uncertainty set. To convert an uncertain optimization problem into a deterministic one, the idea of an objective-wise robust counterpart has been utilized in \cite{15,16}. The set of efficient solutions derived from the objective-wise robust counterpart is a significantly smaller subset of the complete set of robust optimal solutions for UMOPs. In order to identify all robust weakly minimal solutions of UMOPs, recently Ghosh et al. \cite{17} proposed a set-valued optimization view-point of UMOPs instead of just using objective-wise robust counterpart. In \cite{17}, a Newton method is proposed for solving UMOPs. However, in \cite{17}, the objective function for each uncertain scenario is assumed to be locally strong convex, which is restrictive. So, in this study, we focus on developing a quasi-Newton scheme to solve UMOPs without any convexity assumption of the objective. By transforming the considered UMOPs into their min-max counterparts, problems are observed as optimization problems with set-valued maps under the set-ordering relation given by upper less ordering. This reformulated optimization problem with set-valued objectives is solved using a quasi-Newton scheme, which helps to directly obtain robust weakly efficient solutions for the UMOPs under consideration. To achieve this, we construct a sequence of vector optimization problems (VOPs) for the given UMOP by employing the concept of a partition set, as defined in \cite{18}, to solve the pertaining set optimization problem. Based on the ordering cone of the UMOP we find descent direction at every iterate by applying quasi-Newton approach to the constructed sequence of VOPs. This approach aims to find sequences that converge to weakly robust efficient points of the considered UMOPs by systematically analyzing and partitioning the maximal elements within the uncertainty set of objective values across scenarios.

The entire study is organized as follows: In Section \ref{section:2}, we provide notations, definitions, and preliminary results that are used throughout the paper. Section \ref{section 3} reports results on optimality conditions for set-valued optimization, which are later connected with robust weakly minimal solutions of UMOPs under consideration. With the help of these optimality conditions, Section \ref{section:4} derives a quasi-Newton scheme for UMOPs. We establish the well-definedness of the proposed algorithm, supported by the existence of the Armijo step length and detailed convergence analysis. Section \ref{section:5} provides numerical illustrations of the proposed method. Finally, the study is concluded in Section \ref{section 6} by highlighting the key findings and suggesting potential directions for subsequent research.

\section{Notation and Preliminaries}\label{section:2}

The set of all nonempty subsets of $\mathbb{R}^m$ is represented as $P(\mathbb{R}^m)$. The notation $\lVert \cdot \rVert$ denotes either the Euclidean norm for vectors or the spectral norm for matrices, depending on the context. The cardinality of a finite set $A$ is denoted by $|A|$. For any $p\in\mathbb{N}$, we denote  $[p]:= \{1,2,\ldots,p\}$. The notations $\mathcal{N}[\bar{x}, \delta]$ and $\mathcal{N}(\bar{x}, \delta)$, respectively, represent the closed and open balls centred at $\bar{x}\in \mathbb{R}^{n}$ with radius $\delta$. We denote  
\begin{align*}
\mathbb{R}^m_\geqq:= &~ \{z\in\mathbb{R}^m:z\geqq0\}, \text{ where } z\geqq0 \text{ means no component of } z \text{ is negative}, \\
\mathbb{R}^m_\geq:= &~\{z\in\mathbb{R}^m:z\geq0\}, \text{ where } z\geq0 \text{ means } z\geqq0 \text{ but } z\neq0, \\
\text{ and } \mathbb{R}^m_>:= &~ \{z\in\mathbb{R}^m:z>0\}, \text{ where } z>0 \text{ means all components of $z$ are positive}. 
\end{align*}

A nonempty set $K\in P(\mathbb{R}^m)$ is called a cone if $\tau z\in K$ for every $z\in K$ and $\tau\geq 0$. A cone $K$ is called pointed if $K\cap(-K) =\{0\}$, and 
solid if its interior, denoted by $\text{int}(K)$, is nonempty.  

The dual cone of a cone $K$ is the set 
\begin{align*}
K^*:= \{z\in\mathbb{R}^m : z^\top y\geq 0\ \text{for all $y\in K$}\}. 
\end{align*}

The following definitions are well-known in the literature; see \cite{5}, \cite{20}, \cite{23} and references therein.

\begin{definition}
If $K$ is a pointed, closed and convex cone, then it induces a partial order $\preceq$ on $\mathbb{R}^m$ defined as follows: For $y, z \in \mathbb{R}^m$, 
\begin{align*}
z\preceq y \iff y-z\in K.
\end{align*}
In addition, if $K$ is a solid cone, then $K$ induces a strict order $\prec$ on $\mathbb{R}^m$: 
\begin{align*}
z\prec y \iff y-z\in \text{int}(K).
\end{align*}

\end{definition}

A deterministic vector optimization problem is given by
\begin{equation}\label{Pro} 
\underset{x\in S\subseteq \mathbb{R}^n}{\text{min}}~ f(x),
\end{equation}
where $S$ is a nonempty subset of $\mathbb{R}^n$, the function $f:S\to\mathbb{R}^m$ is given by 
$$f(x):= (f_1(x), f_2(x),\ldots, f_m(x))^\top,$$ 
and the optimality concept is given by the following definition with respect to a pointed, closed and convex cone $K$.

\begin{definition}\label{de-effi 2.2}
For the optimization problem \eqref{Pro}, we categorize a point $\bar{x}\in S$ as 
\begin{enumerate}[(i)]
\item efficient if there is no $x\in S$ for which  $f(x)\preceq f(\bar{x})$, and 
\item weakly efficient if for no $x\in S$ we have  $f(x)\prec f(\bar{x})$.
\end{enumerate}  
\end{definition}

In this paper, we consider an UVOP as a family of parameterized optimization problems: 
\begin{align*}
\mathcal{P}(U):=\{\mathcal{P}(\xi):\xi\in U\},
\end{align*}
where for any $\xi\in U$, the  parameterized problem is
\begin{equation}\tag{$\mathcal{P(\xi)}$} \label{Pxi}\underset{x\in S\subseteq \mathbb{R}^n}{\text{min}} F(x,\xi):= \big(f_1(x, \xi), f_2(x,\xi),\ldots, f_m(x, \xi)\big)^\top,
\end{equation}
where $U$ is a nonempty subset of $\mathbb{R}^r$ and $F:\mathbb{R}^n\times U\to\mathbb{R}^m$ is a vector-valued function. The minimization in \eqref{Pxi} is to be understood in the sense of Definition \ref{de-effi 2.2} with respect to a partial order induced by a proper, closed, convex, pointed, solid cone $K$. The set $U$ represents the collection of uncertain (parametric) scenarios. Importantly, note that corresponding to each $\xi\in U$, the problem \eqref{Pxi} is a deterministic vector optimization problem. \\ 

Throughout this paper, we assume the following for the class of problems $\mathcal{P}(U)$.
\begin{enumerate}[(i)]
\item The uncertainties in the problem are given by a finite set of scenarios $\xi$, which collectively form the uncertainty set $U\subseteq\mathbb{R}^r$. Specifically, we denote $ U:= \{\xi_1, \xi_2,\ldots,\xi_p \}.$
\item The feasible set $S$ is assumed to be independent of the uncertainties. The cone $K\subseteq\mathbb{R}^m$ is closed, convex, pointed and solid. Furthermore, let $e\in\text{int}(K)$ be a given element.
\item The vector-valued functions $F(\cdot,\xi_1), F(\cdot,\xi_2),\ldots, F(\cdot,\xi_p):\mathbb{R}^n\to\mathbb{R}^m$ are twice continuously differentiable. 
 \end{enumerate}
For any given $x\in S$, the notation $F_U(x)$ will be used to represent the set $\{F(x,\xi):\xi\in U \}$. It is important to note that for a given $\mathcal{P}(U)$, the mapping $F_U: S\rightrightarrows\mathbb{R}^m$ defines a set-valued map.

In the next definition, we introduce (weakly) robust efficient elements as solutions of a robust counterpart problem to UVOP $\mathcal{P}(U)$ (see \cite{3}).
\begin{definition}\label{De-2.2} 
 For a given UVOP $\mathcal{P}(U)$, a feasible solution $\bar{x}\in S$ is named 
 \begin{enumerate}[(i)]
\item \text{robust efficient} if there is no $x\in S$ such that $F_U(x)\subseteq F_U(\bar{x})-K$, and 
\item weakly robust efficient if there is no $x\in S$ such that $F_U(x) \subset F_U(\bar{x})-\text{int}(K)$. 
\end{enumerate}
\end{definition}


Next, we present the concept of maximal and weakly maximal elements of a set with respect to $\preceq$.
\begin{definition}
Let $C\in P(\mathbb{R}^m)$ and $K$ be a pointed, closed, convex and solid cone.
 \begin{enumerate}[(i)]
\item  The set of maximal elements of $C$ with respect to $K$ is defined as 
\begin{align*}
\text{Max}(C,K) := \{z\in C: (z+K)\cap C = \{z\}\}.
\end{align*}
\item  The set of weakly maximal elements of $C$ with respect to $K$ is defined as  
\begin{align*}
\text{WMax}(C,K) := \{z\in C: (z+\text{int}(K)) \cap C= \emptyset\}.
\end{align*}
 \end{enumerate}
\end{definition}

In the paper \cite{36}, it is shown that the solution concept in Definition \ref{De-2.2} can be formulated using the upper set less relation introduced by Kuroiwa in \cite{37,38,39}. Employing other set less relations (see\cite{23}), it is possible to obtain other concepts of robustness as discussed in \cite{36}.
\begin{definition}\label{De-2.5}\cite{37,38,39}
For a given pointed, closed, convex and solid cone $K\subset\mathbb{R}^m$, the upper set less relation $\preceq^u$ is a binary relation defined on $ {P}(\mathbb{R}^m)$ as follows:
\begin{align*}
    \forall~~ C, D \in  {P}(\mathbb{R}^m) : C \preceq^u D \iff C \subseteq D -K.
\end{align*}
Similarly, if $K$ is solid, the strict upper set less relation $\prec^u$ is a binary relation defined on $ {P}(\mathbb{R}^m)$ by
\begin{align*}
    \forall ~~ C, D \in  {P}(\mathbb{R}^m) : C \prec^u D \iff C \subseteq D - \text{int}(K).
\end{align*}
\end{definition}
Obviously, the (strict) upper set less relation $\preceq^u$ is involved in the formulation of the solution concept in Definition \ref{De-2.2}.

\begin{proposition}\label{2.1}\cite{17}
Let $C\in P(\mathbb{R}^m)$ be compact, and $K$ be a closed, convex and pointed cone. Then, we have 
\begin{align*}
    C-K = \emph{Max}(C,K)-K.
\end{align*}
\end{proposition}

The Gerstewitz scalarizing function will be key in deriving the main results of this study. This function is frequently used in vector optimization problems to convert them into scalar problems. In the following definition, we consider a special case of the Gerstewitz function (see Chapter $5$ in \cite{23}).
\begin{definition}
 Let $K$ be a closed, convex, pointed, and solid cone. For a given element $e\in\text{int}(K)$, the Gerstewitz function $\Theta_e:\mathbb{R}^m\to\mathbb{R}$ associated with $e$ and $K$ is defined by 
 \begin{align}\label{Gerste_f}
\Theta_e(z):= \text{min}\{t\in\mathbb{R}:te\in z+K\}.   
 \end{align}
\end{definition}

A few useful properties of the Gerstewitz scalarizing function $\Theta_e$ are provided in the following proposition, see Chapter $5$ in \cite{23} and references therein.

\begin{proposition}\label{2.2}
Let $y,z \in \mathbb{R}^m$. The Gerstewitz function defined in \eqref{Gerste_f} has the following properties.
\begin{enumerate}[(i)]
\item  $\Theta_e$ \text{is sublinear, i.e.}, $\Theta_e(y+z)\leq \Theta_e(y)+\Theta_e(z)$.
\item $\Theta_e$ is Lipschitz continous on $\mathbb{R}^m$.
\item $\Theta_e$ is monotone, i.e., 
\begin{align*}
\forall~~ y,z \in \mathbb{R}^m: z\preceq y\implies \Theta_e(z)\leq\Theta_e(y)\\
and~\forall~~ y,z \in \mathbb{R}^m: z\prec y\implies \Theta_e(z)<\Theta_e(y).
\end{align*}
\item $\Theta_e$ has the representability property, \emph{i.e.},
\begin{align*}
-~K= \{z\in\mathbb{R}^m:\Theta_e(z)\leq 0\}\\
and~ -\emph{int}(K)=\{z\in\mathbb{R}^m: \Theta_e(z)<0\}.
\end{align*}
\end{enumerate}  
\end{proposition}

\begin{definition}
Consider the UVOP $\mathcal{P}(U)$ and the associated  set-valued mapping $F_U:\mathbb{R}^n\rightrightarrows\mathbb{R}^m$ defined by  
\begin{align*}
F_U(x):= \{F(x,\xi): \xi \in U\} \mbox{ for all } x \in \mathbb{R}^n.    
\end{align*}
Using the upper set less relation $\preceq^u$, we introduce the following robust counterpart problem (SOP) to the given UVOP $\mathcal{P}(U)$:  
\begin{equation}\tag{SOP}\label{SOP}
~\left \{ \begin{aligned}
&  \min    && F_U(x)\\
&\text {subject to}   && x\in S\subseteq\mathbb{R}^n.
   \end{aligned} \right.  
\end{equation}
The robust counterpart problem \eqref{SOP} to $\mathcal{P}(U)$ is a deterministic set-valued optimization problem and the minimization in \eqref{SOP} is to be understood with respect to the set less relation $\preceq^u$, see Definitions \ref{De-2.2} and \ref{De-2.5}. A point $\bar{x}\in\mathbb{R}^n$ is said to be a local weakly efficient of \eqref{SOP}  if there exists a neighbourhood $\mathcal{N}(\bar x, \delta)$ such that 
\begin{equation*}
\nexists ~x\in \mathcal{N}(\bar x, \delta) : F_U(x)\prec^u  F_U(\bar{x}). 
\end{equation*} 
\end{definition}

In this paper, we propose a quasi-Newton method aimed at finding robust weakly efficient solutions for $\mathcal{P}(U)$ through solving the robust counterpart problem \eqref{SOP}.

\section{Optimality Conditions}\label{section 3}
In this section, we report optimality conditions for robust weakly efficient points of $\mathcal{P}(U)$. First, we start this section by mentioning some important index-related set-valued mapping sets. The following notions are introduced in \cite{18}. 

\begin{definition} \label{Def_active_set}
Consider the UVOP $\mathcal{P}(U)$ and the associated set-valued mapping $F_U:\mathbb{R}^n\rightrightarrows\mathbb{R}^m$ as defined above.  Let $x$ be a given element from $S$. 
\begin{enumerate}[(i)]
\item The active index for maximal elements of $F_U(x)$ is $\mathcal{I}: \mathbb{R}^n\rightrightarrows [p]$,  defined  by 
\begin{align*}
 \mathcal{I}(x):= \{i\in[p]:F(x, \xi_i)\in\text{Max}(F_U(x), K)\}.
\end{align*}
\item The active index for weakly maximal elements of  $F_U(x)$ is $\mathcal{I}_w:\mathbb{R}^n\rightrightarrows [p]$, given by
\begin{align*}
\mathcal{I}_w := \{i\in [p]:F(x, \xi_i)\in\text{WMax}(F_U(x), K)\}.
\end{align*}
\item For a vector $v\in\mathbb{R}^m$, the index set $\mathcal{I}_v:\mathbb{R}^n\rightrightarrows [p]$ is defined as 
\begin{align*}
 \mathcal{I}_v(x):= \{i\in\mathcal{I}(x): F(x, \xi_i)= v \}.  
\end{align*}
\end{enumerate}
\end{definition}

\begin{definition}\label{car_m_e}
The mapping $\omega:\mathbb{R}^n\to\mathbb{N}\cup\{0\}$ represents the cardinality of the set of maximal elements of $F_U$ with respect to $K$, i.e., 
\begin{align*}
\omega (x):= |\text{Max}(F_U(x), K)|. 
\end{align*}
Additionally, for a given $\bar{x}\in\mathbb{R}^n$, we define $\bar{\omega}=\omega (\bar{x})$.
\end{definition}

Next, we give the definition of the partition set of a point $x\in\mathbb{R}^n$ in order to systematically identify weakly maximal points of the \eqref{SOP}.

\begin{definition}
Let at a given $x\in\mathbb{R}^n$, $\{v_1^x, v_2^x,\ldots,v_{\omega(x)}^x\}$ be an enumeration of the set Max$(F(x),K)$. The \emph{partition set} at $x$ is the cartesian product 
\begin{align*}
 P_x := \prod_{j =1}^{\omega}\mathcal{I}_{v_j^x}(x). 
\end{align*}
\end{definition}

\begin{lemma}\label{set_to_vector}\cite{17}
   Consider the problem \eqref{SOP}. Let $\bar x$ be a given element from $S$. Denote the cone $\prod_{j =1}^{\bar{\omega}} K$ by $\Tilde{K}$. Furthermore, for every $\beta :=(\beta_1,\beta_2,\ldots,\beta_{\bar{\omega}})\in P_{\bar{x}}$, define a function $\Tilde{F}(\cdot, \xi_{\beta}): \mathbb{R}^n\times U\to\prod_{j=1}^{\bar{\omega}}\mathbb{R}^m$ by 
\begin{align*}
\Tilde{F}(x, \xi_{\beta}):= \left(F(x, \xi_{\beta_1}), F(x, \xi_{\beta_2}),\ldots, F(x, \xi_{\beta_{\bar{\omega}}})\right)^\top.
\end{align*}
Then, $\bar{x}$ is a local weakly efficient solution for \eqref{SOP} if and only if for every $\beta \in P_{\bar{x}}$, $\bar{x}$ is a local weakly efficient solution of the vector optimization problem 
\begin{equation*}
 \preceq_{\Tilde{K}}-\mathrm{(VOP)}(\xi_\beta) ~~~  \left \{ \begin{aligned}
&  \min    && \Tilde{F}(x,\xi_\beta)\\
& \emph{\text {subject to}}   && x\in\mathbb{R}^n.
   \end{aligned} \right.   
\end{equation*}
\end{lemma}

Next, to find a necessary optimality condition for the UVOP $\mathcal{P}(U)$, we recall the concept of stationary point for \eqref{SOP}.

\begin{definition} \cite{18}
 We call $\bar{x}$ as a stationary point for the UVOP $\mathcal{P}(U)$ if there exists a nonempty subset $A\subseteq P_{\bar{x}}$ 
such that the following condition is satisfied:
\begin{align*}
\forall~ \beta \in{A},~ \exists~ \gamma_1, \gamma_2, \ldots, \gamma_{\bar{\omega}} \in K^* : \sum_{j = 1}^{\bar{\omega}} \gradient F(\bar{x}, \xi_{\beta_j}) \gamma_j = 0,~ ~(\gamma_1, \gamma_2, \ldots, \gamma_{\bar{\omega}}) \neq 0.
\end{align*}
\end{definition}

\begin{proposition}  \label{aux_12_01_25_2} \cite{18}
Let $A\subseteq P_{\bar{x}}$ be given. Then, $\bar{x}$ is a stationary point for \eqref{SOP}  with respect to $A$ if and only if 
\begin{align*}
\forall ~\beta\in A, p\in\mathbb{R}^n, \exists~ j\in[\bar{\omega}] \text{ for which } \nabla F(\bar{x}, \xi_{\beta_j})^\top p\notin -\emph{int}(K).
\end{align*}
\end{proposition}

Note that if $\bar{x}$ is a non-stationary point, then for all $j\in[\bar{\omega}]$, there exist $\beta\in P_{\bar{x}}$ and $p\in\mathbb{R}^n$ such that the following condition holds:
\begin{align*}
\nabla F(\bar{x},\xi_{\beta_j})^\top p\in -\text{int}(K),\\
\text{i.e.}, ~\Theta_e(\nabla F(\bar{x},\xi_{\beta_j})^\top p)<0.
\end{align*}

\begin{definition}\cite{18}
If the following two conditions are satisfied at a point $\bar{x}$, then $\bar x$ is called a regular point of $F_U$ of \eqref{SOP}: 
\begin{enumerate}[(i)]
\item Max$(F_U(\bar{x}), K) = \text{WMax} (F_U(\bar{x}), K)$, i.e., the set of maximal elements is equal to the set of weakly maximal elements. 
\item There exists a $\delta> 0$ such that $\omega(x)$ is constant in $\mathcal{N}(x, \delta)$. 
\end{enumerate}
\end{definition}

\begin{lemma}\cite{18}
If $\bar{x}\in\mathbb{R}^n$ is a regular point of $F_U$, then there exists $\delta > 0$ such that for every $x\in \mathcal{N}(\bar x, \delta)$, we have $P_x\subseteq P_{\bar{x}}$ and $\omega(x)=\bar{\omega}$. 
\end{lemma}

\section{Quasi-Newton Method and Its Convergence Analysis}\label{section:4}

For a given initial point $x_0$, for each $i\in [p]$, we construct a sequence $\{B(x_k,\xi_i)\}$ of Hessian approximations of $F(\cdot, \xi_i)$  beginning with an initial matrix $B(x_0,\xi_i)$. For each function $F(\cdot,\xi_i): \mathbb{R}^n \to \mathbb{R}^m$, we use the following quadratic approximation around $x_k$: 
\begin{align*}
q(x,\xi_i):= F(x_k,\xi_i)+\nabla F(x_k,\xi_i)^\top (x-x_k)+\frac{1}{2}(x-x_k)^\top B(x_k,\xi_i)(x-x_k).
\end{align*}
We choose $\{B(x_k,\xi_i)\}$ that satisfies the quasi-Newton equation 
\begin{align*}
    B(x_{k+1}, \xi_i) s_k = y^i_k, ~ k=0,1,2,\ldots,
\end{align*}
where $s_k := x_{k+1}-x_k$ and $y^i_k := \nabla F(x_{k+1}, \xi_i) -\nabla F(x_k, \xi_i)$. To maintain the symmetric and positive definite properties of all terms in the sequence $B(x_k,\xi_i)$, we apply the BFGS update formula starting from a given symmetric and positive definite matrix
$B(x_0,\xi_i)$: 
\begin{align}\label{BFGSF}
 B(x_{k+1}, \xi_i) := B(x_k,\xi_i)- \frac{B(x_k,\xi_i)s_k s_k^\top B(x_k,\xi_i)} {s_k^\top B(x_k,\xi_i)s_k}+ \frac{y^i_k {y^i_k}^\top} {s_k^\top y^i_k}.
\end{align}
From \cite[Section 6.1]{24}, it can be noted that if  $B(x_k,\xi_i)$ is $K$ positive definite, i.e., 
\[p^\top B(x_k, \xi_i) p \in \text{int} (K) \text{ for all } p \in \mathbb{R}^n, \]
then $ B(x_{k+1}, \xi_i)$ will also be $K$ positive definite.

We now present the process to generate a sequence of iterates $\{x_k\}$ starting from an $x_0$, which is built upon the result established in Lemma \ref{set_to_vector}. To proceed from an iterate $x_k$ to the next iterate, we use the usual descent scheme: 
\[x_{k + 1} : = x_k + \tau_k p_k, \]
where $p_k$ is a descent direction of $F_U$ at $x_k$, and $\tau_k>0$. In the following, we formulate a quasi-Newton process to generate $p_k$.

At the current iterate $x_k$, an element $\beta_k$ from the partition set $P_{x_k}$ is selected, and then, a descent direction for (VOP)$(\xi_{\beta_k})$ is found by employing the quasi-Newton method for vector optimization. For a tactful choice of $\beta_k$, we ensure later (in Theorem ) that a decent direction of (VOP)$(\xi_{\beta_k})$ is a descent direction of $F_U$ at $x_k$. The selection of such $\beta_k$ is detailed below. After determining a descent direction, we follow the Armijo-type line search to choose a suitable step size $\tau_k$.

For an appropriate choice of $\beta_k$, define the following parametric family of functionals $\{\varphi_x\}_{x\in\mathbb{R}^n}$, whose elements $\varphi_x: P_x \times \mathbb{R}^n\to \mathbb{R}$ are defined as follows: 
\begin{equation}\label{D-4.1}
\varphi_x(\beta,p) := \underset{j\in[\omega(x)]}{\text{max}}\left\{\Theta_e\left(\nabla F(x,\xi_{\beta_j})^\top p+\frac{1}{2}p^\top B(x,\xi_{\beta_j})p\right)\right\}, ~\beta\in P_x, p\in\mathbb{R}^n. 
\end{equation} 
If for any given  $x\in\mathbb{R}^n$ and $\beta \in P_x$, the matrix $B(x,\xi_{\beta_j})$ is positive definite, then we note that the functional $\varphi_x(\beta,\cdot)$ is strongly convex on $\mathbb{R}^n$ because $\Theta_e$ is sublinear. Therefore, $\varphi_x(\beta,\cdot)$ attains a unique minimum over $\mathbb{R}^n$.

Note that $P_x$ is finite for any $x \in \mathbb{R}^n$. So, $\varphi_x$ attains its minimum over  $P_x\times\mathbb{R}^n$. Hence, the functional $\phi: \mathbb{R}^n\to\mathbb{R}$ given by 
\begin{equation}\label{D-4.2}
    \phi(x) := \underset{(\beta,p)\in P_x\times \mathbb{R}^n}{\text{min}}\varphi_x(\beta,p) 
\end{equation}
is well-defined. In determining a descent direction for $F_U$ at a given nonstationary iterate $x_k$, we choose to employ quasi-Newton method on VOP$(\xi_{\beta_k})$ corresponding to that $\beta_k$ for which we get 
\begin{equation}\label{aux_12_01_25_01}
    (\beta_k, p_k) := \underset{(\beta,p)\in P_{x_k} \times \mathbb{R}^n}{\text{argmin}} \varphi_{x_k}(\beta,p). 
\end{equation}
This is how a $\beta_k$ value is chosen. We show below in Theorem \ref{4.2} that at a nonstationary $x_k$, the $p_k$ determined in \eqref{aux_12_01_25_01} is a descent direction of $F_U$ at $x_k$.  To arrive at this result, we require the following Proposition \ref{p-4.1}. Notice that the formulation of VOP($\beta_k$) is $x_k$-dependent and the direction $p_k$ is obtained from VOP($\xi_{\beta_k}$). So, even if we aim to employ the existing quasi-Newton method for vector optimization \cite{27} on VOP($\xi_{\beta_k}$), the extension is not straightforward because 
VOP($\xi_{\beta_k}$) is determined point-wise, i.e., VOP($\xi_{\beta_k}$) differs once $x_k$ gets differed.

\begin{proposition}\label{p-4.1}
Consider the functions $\varphi_x$ and $\phi$ given in \eqref{D-4.1} and \eqref{D-4.2}, respectively. Suppose $(\bar{\beta},\bar{p})\in P_{\bar{x}}\times \mathbb{R}^n$ satisfies the condition $\phi(\bar{x})=\varphi_{\bar{x}}(\bar{\beta},\bar{p})$. Then, all of the following three statements are equivalent:

\begin{enumerate}[(i)]
\item $\bar{x}$ is a nonstationary point of UVOP $\mathcal{P}(U)$.
\item $\phi(\bar{x})<0$.
\item $\bar{p}\neq0$.
\end{enumerate}
\begin{proof}
$(i)\implies (ii)$. Let $\bar{x}$ be a nonstationary point of the UVOP $\mathcal{P}(U)$. Then, from Proposition \ref{aux_12_01_25_2} we get that for all $j\in[\bar{\omega}]$, there exists a $\beta:= (\beta_1,\beta_2,\ldots,\beta_{\bar{\omega}})\in P_{\bar{x}}$ and $\bar{p}\in\mathbb{R}^n$ with 
\begin{align*}
    \nabla F(\bar{x},\xi_{\beta_j})^\top \bar{p}\in -\text{int}(K), ~\text{i.e.},~ \Theta_e\left(\nabla F(\bar{x},\xi_{\beta_j})^\top \bar{p}\right)<0.
\end{align*}
We notice that 
\allowdisplaybreaks
\begin{align*}
\phi(\bar{x}) = ~&~\underset{(\beta, p)\in P_{\bar{x}}\times\mathbb{R}^n}{\text{min}}\varphi_{\bar{x}}(\beta, p)\\
\leq~&~ \underset{(\beta, p)\in P_{\bar{x}}\times\mathbb{R}^n}{\text{min}}\varphi_{\bar{x}}(\beta, p)\\
\leq~&~\varphi_{\bar{x}}(\beta, t\bar{p}), ~\text{ for some } t>0\\
=~&~\underset{j\in[{\omega(\bar{x})}]}{\text{max}}\big\{\Theta_e\big(\nabla F(\bar{x},\xi_{\beta_j})^\top t \bar{p}+\frac{t}{2} \bar{p}^\top B(\bar{x},\xi_{\beta_j}) t\bar{p}\big)\big\}\\
=~&~t\underset{j\in[{\omega(\bar{x})}]}{\text{max}}\big\{\Theta_e\big(\nabla F(\bar{x},\xi_{\beta_j})^\top \bar{p}+\frac{t}{2} \bar{p}^\top B(\bar{x},\xi_{\beta_j})\bar{p}\big)\big\}, 
 \end{align*}
which is negative for sufficiently small $t>0$.  Consequently, $(ii)$ holds.\\
$(ii)\implies (iii)$. By \eqref{D-4.1}, $\varphi_{\bar{x}}(\beta, 0) = 0$ for any $\beta$ in $P_{\bar{x}}$. Thus, since $\phi(\bar{x})$ is negative, $\bar{p}\neq0$.\\ 
$(iii)\implies (i)$. Assume that $\bar{x}$ is a stationary point. Then, for any $\beta\in P_{\bar{x}}$ and $p\in\mathbb{R}^n$, there exists $j\in[\bar{\omega}]$ such that 
\begin{equation*}
\nabla F(\bar{x}, \xi_{\beta_j})^\top p \not\in -\text{int}(K), ~\text{i.e.}, ~\Theta_e\big(\nabla F(\bar{x}, \xi_{\beta_j})^\top p\big)\geq0.\\
\end{equation*}
Since for all $j\in[\omega(\bar{x})], ~B(\bar{x}, \xi_{\beta_j})$' s are $K$ positive definite matrices, we have  
\begin{align*}
0&\leq \Theta_e\big(\nabla F(\bar{x}, \xi_{\beta_j})^\top p\big)\\
&<\Theta_e\big(\nabla F(\bar{x},\xi_{\beta_j})^\top {p}+\frac{1}{2} {p}^\top B(\bar{x},\xi_{\beta_j}){p}\big)\\
&\leq \underset{j\in[{\omega(\bar{x})}]}{\text{max}}\big\{\Theta_e\big(\nabla F(\bar{x},\xi_{\beta_j})^\top {p}+\frac{1}{2} {p}^\top B(\bar{x},\xi_{\beta_j}){p}\big)\big\}\\
 &=\varphi_{\bar{x}}(\beta, p),   \text{ using }\eqref{D-4.1}\\
\text{i.e.},~ 0~&\leq \varphi_{\bar{x}}(\beta,p) \leq \underset{(\beta,p)\in P_{\bar{x}}\times\mathbb{R}^n}{\text{min}} \varphi_{\bar{x}}(\beta,p)\\
\text{i.e.},~ 0~&\leq\phi(\bar{x})<0, \text{ using }\eqref{D-4.2}.
\end{align*}
This leads to a contradiction, implying that $\bar{x}$ is a nonstationary point for the UVOP $\mathcal{P}(U)$. 
\end{proof}
\end{proposition}
\begin{theorem}\label{Th 4.1}
Consider $S$ as a nonempty open subset of $\mathbb{R}^n$. The function $\phi: S\to \mathbb{R}$, introduced in \eqref{D-4.2}, is continuous.
\end{theorem}
\begin{proof}
The proof is similar to Theorem $4.1$ in \cite{17}.
\end{proof}
Next, we propose a quasi-Newton scheme to solve the robust counterpart problem to $\mathcal{P}(U)$.

\clearpage
\begin{algorithm}[ht]
\caption{ Quasi-Newton method to solve the robust counterpart problem to UMOP $\mathcal{P}(U)$}\label{QNM_algoritm}
\begin{algorithmic}
\Statex {$\textbf{Step 0:}$ Choose $x_0 \in \mathbb{R}^n,~ \gamma \in (0,1)$, and symmetric positive definite matrix $B(x_0,\xi_i)\in\mathbb{R}^{n\times n}$, for each $i\in[p]$. \\ \hspace{0.85cm} Set $k := 0$. Define ${N}:= \{\frac{1}{2^n} : n = 0, 1, 2, \ldots\}.$ }
\Statex {\textbf{Step 1:} Compute $ 
    M_k:= \text{Max}(F_U(x_k), K), ~~P_k:=P_{x_k}, \text{ and } \omega_k:=|\text{Max}(F_U(x_k), K)|. $}
\Statex {\textbf{Step 2:} Find
\begin{align*}
    (\beta^k, p_k) := \underset{(\beta, p) \in P_k \times \mathbb{R}^n}{\text{argmin}}~ \underset{j \in [\omega_k]}{\max} \big\{\Theta_e \big(\gradient F(x_k, \xi_{\beta_j})^\top p + \frac{1}{2} p^\top B(x_k, \xi_{\beta_j}) p \big) \big\}.
\end{align*}}
\Statex{\textbf{Step 3:} \label{step 3} If $p_k = 0$, terminate the process. Otherwise, proceed to Step 4.}
\Statex{\textbf{Step 4:} Choose $\tau_k$ as the largest $\tau \in{N}$ such that
\begin{align*}
   F(x_k+\tau p_k, \xi_{\beta^k_j}) & \preceq F(x_k, \xi_{\beta^k_j}) + \gamma \tau \big( \gradient F(x_k, \xi_{\beta^k_j})^\top p_k + \frac{1}{2} p_k^\top B(x_k, \xi_{\beta^k_j}) p_k \big) ~~\text{ for all } j \in [\omega_k].
\end{align*}}
\Statex{\textbf{Step 5:} Set $x_{k+1} := x_k + \tau_k p_k$.\\
\hspace{0.87cm} Update $B(x_{k+1}, \xi_i)$ for $i\in[p]$ as in \eqref{BFGSF}. Set $k=k+1$,
and go to \bf{Step 1.}}
\end{algorithmic}     
\end{algorithm}

With regard to the well-definedness of the steps of Algorithm \ref{QNM_algoritm}, we note that if Step 2 and Step 4 are well-defined, then all steps are well-defined. Here, we note the following two points:
\begin{enumerate}[(i)]
\item  For the existence of $(\beta^k,p_k)$ in Step $2$, notice  that at any point $x_k\in\mathbb{R}^n$, the map  $ 
p \mapsto (\gradient F(x_k, \xi_{\beta_j})^\top p + \frac{1}{2} p^\top B(x_k, \xi_{\beta_j}) p )$ is a convex function and $[\omega_k]$ is finite. Therefore, a minimum of $\varphi_{x_k}(\beta, p)$ defined in \eqref{D-4.1} always exists. 
\item Existence of a step length $\tau_k$ in Step $4$ is assured by Theorem \ref{4.2}.
\end{enumerate}
Therefore, Algorithm \ref{QNM_algoritm} is well-defined.

\begin{theorem}\label{4.3}
Let $\{x_k\}$ be a sequence of nonstationary points generated by Algorithm  \ref{QNM_algoritm}, $\{p_k\}$ be the corresponding sequence of directions, and $\{x_k\}$ be convergent. Then, the sequence $\{p_k\}$ is bounded.
\end{theorem}

\begin{proof}
The proof is similar to Theorem $4.4$ in \cite{26}. 
\end{proof}

\begin{theorem}\label{4.2}
Let  $\gamma$ be a fixed value in the interval $(0,1)$, and $\bar{x}$ be not a stationary point of \eqref{SOP}. Let $(\bar{\beta},\bar{p})\in P_{\bar{x}}\times\mathbb{R}^n$ be such that $\phi(\bar{x})=\varphi_{\bar{x}}(\bar{\beta},\bar{p})$. Under these conditions, the following statements are true. 
\begin{enumerate}[(i)]
\item \label{S (i)} There exists $\Tilde{\tau}>0$ such that 
\begin{align*}
\forall \tau\in (0,\Tilde{\tau}), j\in[\bar{\omega}]: F(\bar{x}+\tau \bar{p}, \xi_{\bar{\beta_j}})\preceq F(\bar{x}, \xi_{\bar{\beta_j}})+\gamma \tau \big(\nabla F(\bar{x},\xi_{\bar{\beta}})^\top \bar{p}+\frac{1}{2} \bar{p}^\top B(\bar{x},\xi_{\bar{\beta_j}})\bar{p}\big).
\end{align*}
\item \label{4.2(ii)} Additionally, for all $\tau \in (0,\Tilde{\tau})$ and $j\in[\bar{\omega}]$, we have 
\begin{align*}
F_U(\bar{x}+\tau \bar{p}) ~&\preceq^u \big\{F(\bar{x}, \xi_{\bar{\beta_j}})+\gamma \tau \big(\nabla F(\bar{x},\xi_{\bar{\beta}})^\top \bar{p}+\frac{1}{2} \bar{p}^\top B(\bar{x},\xi_{\bar{\beta_j}})\bar{p}\big)\big\}_{j\in[\bar{\omega}]}\\
~&\prec^u F_U(\bar{x}).
\end{align*}
\end{enumerate}
\begin{proof}
\begin{enumerate}[(i)]
\item Suppose that \eqref{S (i)}  is not satisfied. Therefore, there exists a sequence $\{\tau_k\}$ and $j\in[\bar{\omega}]$ such that $\tau_k\rightarrow 0$ and 
\begin{align}\label{S-3}
\forall~ k\in\mathbb{N}: F(\bar{x}+\tau_k \bar{p}, \xi_{\bar{\beta_j}})-F(\bar{x}, \xi_{\bar{\beta_j}})-\gamma \tau_k \big(\nabla F(\bar{x},\xi_{\bar{\beta}})^\top \bar{p}+\frac{1}{2} \bar{p}^\top B(\bar{x},\xi_{\bar{\beta_j}})\bar{p}\big)\notin -K. 
\end{align}
Multiply by $\frac{1}{\tau_k}$ in \eqref{S-3} for each $k\in\mathbb{N}$ to obtain 
\begin{align}\label{S-4}
\forall~ k\in\mathbb{N}: \frac{F(\bar{x}+\tau_k \bar{p}, \xi_{\bar{\beta_j}})-F(\bar{x}, \xi_{\bar{\beta_j}})}{\tau_k}-\gamma \big(\nabla F(\bar{x},\xi_{\bar{\beta}})^\top \bar{p}+\frac{1}{2} \bar{p}^\top B(\bar{x},\xi_{\bar{\beta_j}})\bar{p}\big)\notin -K.
\end{align}
Now taking the limit $k\rightarrow \infty$ in \eqref{S-4}, we get 
\begin{align}\label{S-5}
\nabla F(\bar{x}, \xi_{\bar{\beta_j}})^\top \bar{p}-\gamma\big(\nabla F(\bar{x},\xi_{\bar{\beta}})^\top \bar{p}+\frac{1}{2} \bar{p}^\top B(\bar{x},\xi_{\bar{\beta_j}})\bar{p}\big)\notin -K.  
\end{align}
Note that $\bar{x}$ is not a stationary point of \eqref{SOP} and $(\bar{\beta},\bar{p})\in P_{\bar{x}}\times \mathbb{R}^n$ is such that $\phi(\bar{x})=\varphi_{\bar{x}}(\bar{\beta}, \bar{p})$. Therefore, by Proposition \ref{p-4.1}, we obtain $\bar{p}\neq 0$ and $\phi(\bar{x})<0$. So, by the continuity of $\phi$, (see Theorem \ref{Th 4.1}) there exists $\delta>0$ such that 
\begin{align}\label{aux_13_01_1}
    \phi(x)\leq \frac{1}{2}\phi(\bar{x}) < 0 \text{ for all } x\in\mathcal{N}[\bar{x}, \delta].
\end{align}

\noindent
Let us consider a sequence $\{p_k\}$ that lies in $\mathcal{N}[\bar{x}, \delta]$, i.e., $\|p_k - \bar x\| \le \delta$, and converges to $\bar p$. By Taylor expansion, we obtain
\begin{align*}
F(\bar x+\tau_k p_k,\xi_{\beta_j})= F(\bar x,\xi_{\beta_j})+\tau_k \nabla F(\bar x,\xi_{\beta_j})^\top p_k+\circ_j(\tau_k p_k,x)e \text{ for all } j\in[\omega(\bar x)],
\end{align*}
where $\underset{k\rightarrow\infty}{\lim}\frac{\circ_j(\tau_k p_k, \bar x)e}{\tau_k\lVert p_k \rVert}= (0,0,\ldots,0)^\top$. Since the sequence $\{p_k\}$ is bounded, we have 
\begin{align*}
\lim_{k\to\infty}\frac{\circ_j(\tau_k p_k,\bar x)e}{\tau_k}= (0,0,\ldots,0)^\top.
\end{align*}
Now, we have $\nabla F(\bar x,\xi_{\beta_j})^\top p_k \preceq\nabla F(\bar x,\xi_{\beta_j})^\top p_k +\frac{1}{2} p_k^\top B(\bar x,\xi_{\beta_j})p_k$ since  $p_k^\top B(\bar x,\xi_{\beta_j})p_k$ is positive definite for all $j\in[\omega(\bar x)]$. \\ 

\noindent
So, if $\{x_k\}$ is a sequence generated be Algorithm \ref{QNM_algoritm} and converges to $\bar x$, then for all $j\in[\omega(\bar x)]$ we have 
\[ 
\resizebox{0.85\hsize}{!}{$F(x_k+\tau_k p_k,\xi_{\beta_j^k})\preceq F(x_k,\xi_{\beta_j^k})+\tau_k\big(\nabla F(x_k,\xi_{\beta_j^k})^\top p_k +\frac{1}{2}p_k^\top B(x_k,\xi_{\beta_j^k})p_k\big)+\circ_j(\tau_k p_k,x_k)e.$}
\]

\noindent
Therefore,
\begin{align*}
&\frac{F(x_k+\tau_k p_k,\xi_{\beta_j^k})-F(x_k,\xi_{\beta_j^k})}{\tau_k}\\
\preceq &~ \gamma\big(\nabla F(x_k,\xi_{\beta_j^k})^\top p_k+\frac{1}{2} p_k^\top ~B(x_k,\xi_{\beta_j^k})p_k\big)\\
&+\big[(1-\gamma)\big(\nabla F(x_k,\xi_{\beta_j^k})^\top~ p_k+\frac{1}{2} p_k^\top B(x_k,\xi_{\beta_j^k})p_k\big)+\frac{\circ_j(\tau_k p_k,x_k)e}{\tau_k}\big].
\end{align*}
Taking limit $k\to\infty$ in the above inequality, we have from \eqref{aux_13_01_1} that 
\begin{align*}
~&~\nabla F(\bar{x}, \xi_{\bar{\beta}_j})^\top \bar{p}\preceq \gamma \big(\nabla F(\bar{x},\xi_{\bar{\beta}_j})^\top\bar{p}+\frac{1}{2}\bar{p}^\top B(\bar{x},\xi_{\bar{\beta}_j})\bar{p}\big),\\
\text{i.e.},~&~\nabla F(\bar{x}, \xi_{\bar{\beta}_j})^\top \bar{p}- \gamma \big(\nabla F(\bar{x},\xi_{\bar{\beta}_j})^\top\bar{p}+\frac{1}{2}\bar{p}^\top B(\bar{x},\xi_{\bar{\beta}_j})\bar{p}\big)\in - \text{int}(K) \text{ for all } j\in[\omega(\bar{x})],
\end{align*}
which is contradictory to \eqref{S-5}. Hence, the statement \eqref{S (i)} is proved.
\item As $\bar{x}$ is not a stationary point, then by Proposition \ref{p-4.1}, we obtain 
\allowdisplaybreaks
\begin{align*}
~&~\phi(\bar{x})<0\\
\implies ~&~\varphi_{\bar{x}}(\bar{\beta}, \bar{p})<0\\
\implies ~&~\underset{j \in [\bar{\omega}]}{\max} \big\{\Theta_e \big(\gradient F(\bar{x}, \xi_{\bar{\beta_j}})^\top \bar{p} + \frac{1}{2} \bar{p}^\top B(\bar{x}, \xi_{\bar{\beta_j}}) \bar{p} \big) \big\}<0\\
\implies ~&~ \text{ for any } j\in[\bar{\omega}],\Theta_e \big(\gradient F(\bar{x}, \xi_{\bar{\beta_j}})^\top \bar{p} + \frac{1}{2} \bar{p}^\top B(\bar{x}, \xi_{\bar{\beta_j}}) \bar{p} \big)<0\\
\implies~&~\gradient F(\bar{x}, \xi_{\bar{\beta_j}})^\top \bar{p} + \frac{1}{2} \bar{p}^\top B(\bar{x}, \xi_{\bar{\beta_j}}) \bar{p}\in -\text{int}(K).
\end{align*}
Using this in inequality (\ref{S (i)}), we have \\
\begin{align}\label{4.7}
\forall \tau\in(0,\Tilde{\tau}), j\in[\bar{\omega}] : F(\bar{x}+\tau \bar{p},\xi_{\bar{\beta_j}})\preceq ~&~F(\bar{x},\xi_{\bar{\beta_j}})+\gamma \tau \big(\gradient F(\bar{x}, \xi_{\bar{\beta_j}})^\top \bar{p} + \frac{1}{2} \bar{p}^\top B(\bar{x}, \xi_{\bar{\beta_j}}) \bar{p}\big)\nonumber\\
\prec ~&~ F(\bar{x},\xi_{\bar{\beta_j}}).
\end{align}
As Max$(F_U(\bar x + \tau \bar p), K) = \{F(\bar x + \tau \bar p, \xi_{\bar \beta_j}): j \in [\bar \omega]\}$, it follows from Proposition \ref{2.1} that 
\begin{align*}
F_U(\bar{x}+\tau \bar{p}) &  \subseteq \left\{F\left(\bar{x}+\tau \bar{p},\xi_{\bar{\beta_j}}\right)\right\}_{j\in[\bar{\omega}]}-K\\
 & \overset{\eqref{4.7}}
 {\subseteq} \left\{F(\bar{x},\xi_{\bar{\beta_j}})+\gamma \tau \left(\gradient F\big(\bar{x}, \xi_{\bar{\beta_j}}\big)^\top \bar{p} + \frac{1}{2} \bar{p}^\top B(\bar{x}, \xi_{\bar{\beta_j}}) \bar{p}\right)\right\}_{j\in[\bar{\omega}]}-K\\
& \subseteq \left\{F(\bar{x}, \xi_{\bar{\beta_j}})\right\}_{j\in[\bar{\omega}]}-\text{int}(K)-K ~\text{ by \text{Statement (\ref{S (i)})}} \\
& \subseteq \big\{F(\bar{x},\xi_{\bar{\beta_1}}),F(\bar{x},\xi_{\bar{\beta_2}}),\ldots, F(\bar{x},\xi_{\bar{\beta}_{\bar{\omega}}})\big\}-\text{int}(K)\\
& \subseteq F_U(\bar{x})-\text{int}(K).
\end{align*}
\end{enumerate}
Hence, the statement \eqref{4.2(ii)} is proved.
\end{proof}
\end{theorem}

\begin{theorem}\label{4.4}
Suppose $\{x_k\}$ is a sequence produced by  Algorithm \ref{QNM_algoritm} and $\bar{x}$ is an accumulation point for the sequence $\{x_k\}$. Additionally, suppose that $\bar{x}$ is a regular point for $F_U$. Then, $\bar{x}$ is a stationary point of \eqref{SOP}.
\begin{proof}
Consider the functional $\Upsilon:P(\mathbb{R}^m)\to\mathbb{R}\cup\{+\infty\}$ defined as 
\begin{align*}
 \Upsilon(A): =\underset{z\in A}{\text{ sup }} \Theta_e(z),~ A\in P(\mathbb{R}^m). 
\end{align*}
First, we show that
\begin{align*}
\forall~ k\in\mathbb{N}\cup\{0\} : (\Upsilon \circ F_U)(x_{k+1})\leq (\Upsilon \circ F_U)(x_k) +\gamma \tau_k \phi(x_k).
\end{align*}
Due to the monotonicity property of $\Theta_e$ in Proposition \ref{2.2} (iii),  the function $\Upsilon$ is also monotone with respect to the $\preceq^u$, i.e., 
\begin{align*}
\forall~ A, B\in P(\mathbb{R}^m) : A\preceq^u B \implies \Upsilon(A)\leq\Upsilon(B).
\end{align*}
Now in view of \eqref{4.2(ii)} of Theorem \ref{4.2}, for all $k\in\mathbb{N}\cup \{0\}$,
\begin{align*}
F_U(x_k+\tau_k p_k) \preceq^u \big\{F(x_k,\xi_{\beta_j^k}) +\gamma\tau_k\big(\nabla F(x_k, \xi_{\beta_j^k})^\top p_k +\frac{1}{2}p_k^\top B(x_k, \xi_{\beta_j^k})p_k\big)\big\}_{j\in[\omega_k]}.
\end{align*}
Therefore, by applying the monotonicity of $\Upsilon$ and the sublinearity of $\Theta_e$, we derive that for any $k\in\mathbb{N}\cup\{0\}$,
\begin{align*}
(\Upsilon \circ F_U)(x_{k+1}) &\leq \underset{j\in[\omega_k]} {\text{max}}\big\{\Theta_e\big(F(x_k,\xi_{\beta_j^k}) +\gamma\tau_k\big(\nabla F(x_k, \xi_{\beta_j^k})^\top p_k +\frac{1}{2}p_k^\top B(x_k, \xi_{\beta_j^k})p_k\big)\big)\big\}\\
&\leq \underset{j\in[\omega_k]}{\text{max}}\big\{\Theta_e\big(F(x_k,\xi_{\beta_j^k})\big) +\gamma\tau_k \Theta_e\big(\nabla F(x_k, \xi_{\beta_j^k})^\top p_k +\frac{1}{2}p_k^\top B(x_k, \xi_{\beta_j^k})p_k\big)\big\}\\
& \leq \underset{j\in[\omega_k]}{\text{max}}\Theta_e\big(F(x_k,\xi_{\beta_j^k})\big) +\gamma\tau_k \underset{j\in[\omega_k]}{\text{max}}\Theta_e \big(\nabla F(x_k, \xi_{\beta_j^k})^\top p_k +\frac{1}{2}p_k^\top B(x_k, \xi_{\beta_j^k})p_k\big)\\
& = (\Upsilon \circ F_U)(x_k) +\gamma \tau_k \underset{j\in[\omega_k]}{\text{max}}\Theta_e \big(\nabla F(x_k, \xi_{\beta_j^k})^\top p_k +\frac{1}{2}p_k^\top B(x_k, \xi_{\beta_j^k})p_k\big)\\
\implies -\gamma \tau_k \underset{j\in[\omega_k]}{\text{max}}\Theta_e &\big(\nabla F(x_k, \xi_{\beta_j^k})^\top p_k +\frac{1}{2}p_k^\top B(x_k, \xi_{\beta_j^k})p_k\big)\leq (\Upsilon \circ F_U)(x_k)-(\Upsilon \circ F_U)(x_{k+1}).
\end{align*}
By summing the above inequality for $k=0,1,2,\ldots, \bar{k} $, we have 
\begin{align*}
-\gamma \sum_{k=0}^{\bar{k}}\tau_k \underset{j\in[\omega_k]}{\text{max}}\Theta_e \big(\nabla F(x_k, \xi_{\beta_j^k})^\top p_k +\frac{1}{2}p_k^\top B(x_k, \xi_{\beta_j^k})p_k\big)\leq (\Upsilon \circ F_U)(x_0)-(\Upsilon \circ F_U)(x_{\bar{k}+1}) \\ 
\text{i.e.},~-\sum_{k=0}^{\bar{k}}\tau_k \underset{j\in[\omega_k]}{\text{max}}\Theta_e \big(\nabla F(x_k, \xi_{\beta_j^k})^\top p_k +\frac{1}{2}p_k^\top B(x_k, \xi_{\beta_j^k})p_k\big)\leq \frac{(\Upsilon \circ F_U)(x_0)-(\Upsilon \circ F_U)(x_{\bar{k}+1})}{\gamma}.
\end{align*}
As $\{x_k\}$ is a sequence of nonstationary points,
\begin{align}\label{4.8} 
\nabla F(x_k, \xi_{\beta_j^k})^\top p_k +\frac{1}{2}p_k^\top B(x_k, \xi_{\beta_j^k})p_k\in - \text{int}(K).
\end{align}
Specifically, by applying Proposition \ref{2.2} in \eqref{4.8}, we find that for all $k\in\mathbb{N}\cup \{0\}, j\in[\omega_k]$, the following holds
\begin{align*}
\Theta_e \big(\nabla F(x_k, \xi_{\beta_j^k})^\top p_k +\frac{1}{2}p_k^\top B(x_k, \xi_{\beta_j^k})p_k\big)<0.
\end{align*}
Therefore, we conclude that
\begin{align*}
 0< -\sum_{k=0}^{\bar{k}}\tau_k \underset{j\in[\omega_k]}{\text{max}}\Theta_e \big(\nabla F(x_k, \xi_{\beta_j^k})^\top p_k +\frac{1}{2}p_k^\top B(x_k, \xi_{\beta_j^k})p_k\big)\leq \frac{(\Upsilon \circ F_U)(x_0)-(\Upsilon \circ F_U)(x_{\bar{k}+1})}{\gamma}.
\end{align*}
Now, by taking the limit of the previous inequality as $\bar{k} \to \infty$, we deduce that 
\begin{align*}
0\leq -\sum_{k=0}^{\infty} \tau_k \underset{j\in[\omega_k]}{\text{max}}\Theta_e \big(\nabla F(x_k, \xi_{\beta_j^k})^\top p_k +\frac{1}{2}p_k^\top B(x_k, \xi_{\beta_j^k})p_k\big)<\infty. 
\end{align*}
In particular, this implies that 
\begin{align}
 \nonumber & \lim_{k\to\infty}\tau_k \underset{j\in[\omega_k]}{\text{max}}\Theta_e\big(\nabla F(x_k, \xi_{\beta_j^k})^\top p_k +\frac{1}{2}p_k^\top B(x_k, \xi_{\beta_j^k})p_k\big)=0\\
\text{i.e.}, ~~ & \lim_{k\to\infty}\tau_k\phi(x_k)=0.\label{4.9}
\end{align}
However, since $\bar{x}$ is an accumulation point of the sequence $\{x_k\}_{k\geq 0}$, there exists a subsequence $\mathcal{K}$ in $\mathbb{N}$ such that $x_k\xrightarrow{\mathcal{K}}\bar{x}$. Since there are only a finite number of subsets of $[p]$ and $\bar{x}$ is regular for $F_U$, we can utilize Lemma 4.2 of \cite{18}. From this, it follows that, without any loss of generality, there exists a subset $A\subseteq P_{\bar{x}}$ and an element $\bar{\beta}\in A$ such that for large $k$,
\begin{align*}
\omega_k = \bar{\omega}, P_{x_k} = A, \text{ and } \beta^k = \bar{\beta}. 
\end{align*}
Furthermore, since the sequence $\{\tau_k\}$ and $\{p_k\}$ are bounded by Theorem \ref{4.3}, we can find $\bar{\tau}>0$, $\bar{p}\in\mathbb{R}^n$ such that
\begin{align*}
\tau_k\xrightarrow{k\in\mathcal{K}}\bar{\tau} \text{ and } p_k\xrightarrow{k\in\mathcal{K}}\bar{p}.
\end{align*}
Assume that $\bar{x}$ is nonstationary, according to  Proposition \ref{p-4.1}, it implies that $\phi(\bar{x})<0$ and $\bar{p}\neq 0$. By applying Theorem \ref{4.2} (ii), it follows that there exists an integer $q$ such that 
\begin{align}\label{aux_13_01_25_02}
\Upsilon\circ\big(F_U(\bar{x}+2^{-q}\bar{p})-F_U(\bar{x})\big)<\gamma ~2^{-q}\phi(\bar{x}).
\end{align}
Since $\phi$ and $\Upsilon$ are continuous within their respective domains, we have
\begin{align}\label{4.10}
\lim_{k\to\infty}p_k = \bar{p} \text{ and } \lim_{k\to\infty}\phi(x_k)= \phi(\bar{x})<0.
\end{align}
For large enough $k$, we have from \eqref{aux_13_01_25_02} that 
\begin{align*}
(\Upsilon\circ F_U)(x_k+2^{-q}p_k)-(\Upsilon\circ F_U)(x_k)<\gamma ~2^{-q}\phi(x_k).
\end{align*}
Considering the definition of $\Upsilon$ and Step 4 of Algorithm \ref{QNM_algoritm}, this implies that for sufficiently large $k$, $\tau_k\geq 2^{-q}$. Therefore, by taking into account the second limit in \eqref{4.10}, we deduce that 
\begin{align*}
\liminf_{k\to\infty}\tau_k|\phi(x_k)|>0, 
\end{align*}
which is contradictory to \eqref{4.9}. Thus, the desired result follows.
\end{proof}
\end{theorem}

Before we start analysing the convergence properties of the proposed quasi-Newton Algorithm \ref{QNM_algoritm}, we recall some results confirming the superlinear convergence of the BFGS method. The first result, based on \cite{19}, states that if $\{x_k\}$ is a sequence of nonstationary points converging to $\bar{x}$, the following relation holds
\begin{align*}
\lim_{k\to\infty} \frac{\left\lVert \left(B(x_k)-\nabla^2 F({x}_k)\right) s_k \right\rVert}{\big\lVert s_k \big\rVert} = 0,
\end{align*}
where $F$ is twice continuously differentiable such that $\nabla F(\bar{x}) = 0$ and $\nabla^2 F(\bar{x})$ is positive definite, $s_k := x_{k+1}-x_k = \tau_k p_k$. Moreover, $B_k$ is a BFGS approximation at iteration $k$. 
Now, from \cite{27}, and in view of the above result, for some $\epsilon>0$, there exists $k_0\in\mathbb{N}$ such that for all $k\geq k_0$, the following condition holds

\begin{align*}
\lim_{k\to\infty} \frac{\big\lVert (B(x_k, \xi_{\beta_j^k})-\nabla^2 F(\bar{x}, \xi_{\bar{\beta_j}^k})) p_k \big\rVert}{\big\lVert p_k \big\rVert} < \epsilon \text{ for every } j\in[\omega(x_k)].
\end{align*}

\begin{lemma}\label{l}\cite{18,22}
Let $\bar{x}$ be a regular point of the function $F_U$. 
\begin{enumerate}[(i)]
\item Then, there exists $\delta > 0$, such that for every $x\in \mathcal{N}(\bar x, \delta)$, we have $\omega(x)$ = $\omega(\bar{x})$.
\item Let $\epsilon>$0 and $\delta>0$ be such that for any $x,y\in S$, with $\big\lVert y-x \big\rVert<\delta$, the following condition holds:
\begin{align}
\big\lVert \nabla^2 F(y,\xi_{\beta_j}) - \nabla^2 F(x,\xi_{\beta_j}) \big\rVert <\frac{\epsilon}{2}\text{ for all }j=1,2,\ldots,[\omega(\bar{x})].
\end{align}
Considering this assumption, for any $x,y\in S$ such that $\big\lVert y-x \big\rVert<\delta$ we have  that 
\begin{enumerate}[(a)]
\item 
\begin{align}\label{L_SC_eq1} 
\big\lVert \nabla F(y,\xi_{\beta_j})-[\nabla F(x,\xi_{\beta_j})+\nabla^2 F(x,\xi_{\beta_j})(y-x)] \big\rVert <\frac{\epsilon}{2} \big\lVert y-x \big\rVert, \text{ and }
\end{align}
\item 
\begin{align}\label{l-4.4}
 ~&~\bigg\lVert F(y,\xi_{\beta_j})-\big[ F(x,\xi_{\beta_j})+\nabla F(x,\xi_{\beta_j})^\top (y-x)+\frac{1}{2} (y-x)^\top \nabla^2F(x,\xi_{\beta_j})(y-x)\big] \bigg\rVert <\frac{\epsilon}{4} \big\lVert y-x \big\rVert^2 \\
 ~&~\text{ for all } j= 1,2,\ldots,[\omega(\bar{x})]. \nonumber
\end{align}
\end{enumerate}
\end{enumerate}
\end{lemma}

Next, we modify  Lemma \ref{l} to assess the approximation error, using the BFGS approximation of the second-order derivative term Hessian.
\begin{lemma}
Let $V\subset S$ be a convex subset, and $\epsilon > 0$. Let $V$ be such that there exists a $\delta> 0$ for which any $x,y\in V$ holds $\lVert y-x \rVert<\delta$. Let $\{x_k\}$ be a sequence generated by Algorithm \ref{QNM_algoritm}. Assume that for the given $\epsilon>0$, there exists $k_0\in\mathbb{N}$ such that for all $k\geq k_0$, we have 

\begin{align}\label{4.15}
\frac{\big\lVert(\nabla^2 F(x_k,\xi_{\beta_j^k})-B(x_k, \xi_{\beta_j^k}))(y-x_k) \big\rVert}{\big\lVert y-x_k \big\rVert} <\frac{\epsilon}{2}.
\end{align}
Then, for any $x_k$, $k\geq k_0$, and $y\in V$ with $\big\lVert y-x_k \big\rVert<\delta$, we have 
\begin{align}
  \big\lVert \nabla F(y,\xi_{\beta_j})-[\nabla F(x_k,\xi_{\beta_j^k})+B(x_k,\xi_{\beta_j^k})(y-x_k)] \big\rVert <\epsilon \big\lVert y-x_k \big\rVert,
\end{align}
and
\begin{align}\label{l-4.7}
\big\lVert F(y,\xi_{\beta_j})-[ F(x_k,\xi_{\beta_j^k})+\nabla F(x_k,\xi_{\beta_j^k})^\top(y-x_k)+\frac{1}{2} (y-x_k)^\top B(x_k,\xi_{\beta_j^k})(y-x_k)] \big\rVert <\frac{\epsilon}{2} \big\lVert y-x_k \big\rVert^2.
\end{align}
\begin{proof}
From \eqref{L_SC_eq1} of Lemma \ref{l}, for every $j\in\ [\omega(x)]$ and $x,y\in S$ such that $\big\lVert y-x \big\rVert<\delta$, we have
\begin{align*}
 ~&~\big\lVert \nabla F(y,\xi_{\beta_j})-[\nabla F(x_k,\xi_{\beta_j^k})+B(x_k,\xi_{\beta_j^k})(y-x_k)] \big\rVert\\
\leq~&~\big\lVert \nabla F(y,\xi_{\beta_j})-\nabla F(x_k,\xi_{\beta_j^k})-\nabla^2 F(x_k,\xi_{\beta_j^k})(y-x_k) \big\rVert+\big\lVert (\nabla^2 F(x_k,\xi_{\beta_j^k})-B(x_k,\xi_{\beta_j^k}))(y-x_k) \big\rVert \\
\leq~&~\frac{\epsilon}{2} \big\lVert y-x_k \big\rVert +\frac{\epsilon}{2} \big\lVert y-x_k \big\rVert,~\text{using}~\eqref{L_SC_eq1}~\text{and}~\eqref{4.15}\\
 \leq~&~\epsilon \big\lVert y-x_k \big\rVert. 
\end{align*}
Similarly, by \eqref{l-4.4}, 
\begin{align*}
~&~\big\lVert F(y,\xi_{\beta_j})-[ F(x_k,\xi_{\beta_j^k})+\nabla F(x_k,\xi_{\beta_j^k})^\top(y-x_k)+\frac{1}{2} (y-x_k)^\top B(x_k,\xi_{\beta_j^k})(y-x_k)] \big\rVert \\
\leq ~&~\big\lVert F(y,\xi_{\beta_j})-F(x_k,\xi_{\beta_j}^k)-\nabla F(x_k,\xi_{\beta_j}^k)^\top (y-x_k)-\frac{1}{2} (y-x_k)^\top \nabla^2F(x_k,\xi_{\beta_j}^k)(y-x_k) \big\rVert \\
~&~ +\big\lVert \frac{1}{2} (y-x_k)^\top (\nabla^2 F(x_k,\xi_{\beta_j^k})-B(x_k,\xi_{\beta_j^k}))(y-x_k) \big\rVert\\
<~&~\frac{\epsilon}{4} \big\lVert y-x_k \big\rVert^2 +\frac{\epsilon}{4} \big\lVert y-x_k \big\rVert^2 \\
<~&~\frac{\epsilon}{2} \big\lVert y-x_k \big\rVert^2, 
\end{align*}
which confirms \eqref{l-4.7} and the proof is complete.
\end{proof}
\end{lemma}

\begin{theorem}\label{quadratic}\emph{(Superlinear convergence).}
Let $\{x_k\}$ be a sequence of nonstationary points generated by Algorithm \ref{QNM_algoritm} and $\bar x$ be one of its accumulation points. Additionally, let $\bar{x}$ be a regular point for $F_U$. Suppose that $S$ be a neighborhood of $\bar{x}$ and there exists $a>0,b>0,\delta>0 \text{ and } \epsilon>0$ for which the following conditions holds:
\begin{enumerate}[(i)]
\item $aI\leq B(x,\xi_j)\leq bI$ for all $j\in[\omega(x)]$,
\item $\big\lVert \nabla^2F(x,\xi_{\beta_j})-\nabla^2F(y,\xi_{\beta_j}) \big\rVert <\epsilon$ for all $x,y\in S$ with $\big\lVert x-y \big\rVert < \delta$,
\item $\big\lVert (\nabla^2F(x_k,\xi_{\beta_j^k})- B(x_k,\xi_{\beta_j^k}))(y-x_k))\big\rVert<\epsilon$  for all $x_k,y\in S$ with $\big\lVert x_k-y\big\rVert<\delta$, and
\item $\epsilon\leq a(1-\gamma)$. 
\end{enumerate} 
Then, $\tau_k=1$, for sufficiently large $k\in\mathbb{N}$, and the sequence $\{x_k\}$ converges superlinearly to $\bar{x}\in\mathbb{R}^n$.
\end{theorem} 
\begin{proof}
Theorem \ref{4.4} established the convergence of the sequence $\{x_k\}$ to a stationary point $\bar{x}$. Given that $F(\cdot,\xi_{\beta_j})$ functions are twice continuously differentiable, it follows that for any $\epsilon>0$, there exists $\delta_\epsilon>0$ 
such that 
\begin{align*}
\mathcal{N}(\bar{x},\delta_\epsilon)\subset S \text{ and } \big\lVert B(x,\xi_{\beta_j})-B(y,\xi_{\beta_j})\big\rVert <\epsilon \text{ for all } x,y\in \mathcal{N}(\bar{x},\delta_\epsilon).
\end{align*}
Now, for $x\in S$, $\lambda_j\geq0$ and $\sum_{j=1}^{\omega(x)}\lambda_j=1 $, 
we define a function $\Phi_{\lambda}: S\times\mathbb{R}^n\to\mathbb{R}^m$ by
\begin{align*}
\Phi_{\lambda}(x,p):= \sum_{j=1}^{[\omega(x)]}\lambda_j\nabla F(x,\xi_{\beta_j})^\top p +\frac{1}{2}\sum_{j=1}^{[\omega(x)]}\lambda_j p^\top B(x, \xi_{\beta_j})p.
\end{align*}
Note that for any $\beta^k\in P_{x_k}$ and $j\in[\omega(x_k)]$, the matrix $ B(x_k, \xi_{\beta^k_j})$ is $K$ positive definite.  
Furthermore, for any $x_k\in\mathbb{R}^n$ and $\beta^k\in P_{x_k}$, the set $P_{x_k}$ is finite. Therefore, the function $\Phi_{\lambda}(x,\cdot)$ is strictly convex in $\mathbb{R}^n$, and hence the function attains its minimum. Observe that $\Phi_{\lambda}$ will attain minimum value when 
\begin{align}\label{4.18}
~&~\sum_{j=1}^{[\omega(x_k)]}\lambda_j\nabla F(x_k,\xi_{\beta_j^k}) +\sum_{j=1}^{[\omega(x_k)]}\lambda_j B(x_k, \xi_{\beta_j^k})p_k=0.\\
\nonumber \implies p_k=&-\biggl[\sum_{j=1}^{[\omega(x_k)]}\lambda_j B(x_k, \xi_{\beta_j^k})\biggl]^{-1}\sum_{j=1}^{[\omega(x_k)]}\lambda_j\nabla F(x_k,\xi_{\beta_j^k})\\
\nonumber \implies p_k\leq & -\frac{1}{a}\sum_{j=1}^{[\omega(x_k)]}\lambda_j\nabla F(x_k,\xi_{\beta_j^k})~ \text{ since } aI\leq B(x_k,\xi_{\beta^k_j})\\
\implies p_k\leq & -\frac{1}{a} \underset{j\in[\omega(x_k)]}{\text{max}}\nabla F(x_k,\xi_{\beta_j^k}).\label{1kg}
\end{align}
Note that $\sum_{j=1}^{[\omega(x)]} \lambda_j b_j\leq\text{max} \{b_1,b_2,\ldots,b_{[\omega(x)]}\}$ holds, where $(\lambda_1,\lambda_2,\ldots,\lambda_{[\omega(x)]})\in\mathbb{R}^{\omega(x)}_{+}$ with $\sum_{j=1}^{[\omega(x)]} \lambda_j =1$ and $b_j$'s  in $\mathbb{R}$. This identity is applied in obtaining \eqref{1kg}. For the sequence $\{x_k\}$, using Lemma \ref{l}, we have 
\allowdisplaybreaks 
\begin{align}
\nonumber&\left\lVert \sum_{j=1}^{[\omega(x_{k+1})]}\lambda_j\nabla F(x_k+p_k,\xi_{\beta_j^k})-\biggl[\sum_{j=1}^{[\omega(x_k)]}\lambda_j\nabla F(x_k,\xi_{\beta_j^k})+\sum_{j=1}^{[\omega(x_k)]}\lambda_j B(x_k,\xi_{\beta_j^k})p_k\biggl]\right\rVert\leq\epsilon\big\lVert p_k \big\rVert\\
\nonumber \text{i.e.},~ &\left\lVert \sum_{j=1}^{[\omega(x_{k+1})]}\lambda_j\nabla F(x_k+p_k,\xi_{\beta_j^k})\right\rVert \leq\epsilon \big\lVert p_k \big\rVert\\
\nonumber \text{i.e.},~&\bigg\lVert\underset{j\in[\omega(x_{k+1})]}{\text{max}}\nabla F(x_k+p_k,\xi_{\beta_j^k})\bigg\rVert \leq\epsilon\big\lVert p_k \big\rVert\\
\text{i.e.},~& \big\lVert p_{k+1} \big\rVert \leq \frac{\epsilon}{a}\big\lVert p_k \big\rVert ~\text{ using }\eqref{1kg} \label{4.20}.
\end{align}
As the sequence $\{x_k\}$ converges to $\bar{x}$, there exists $k_{\epsilon}\in \mathbb{N}$ such that for all $k\geq k_{\epsilon}$, we have 
\begin{align*}
x_k,~ x_k+p_k \in\mathcal{N}(\bar{x},\delta_{\epsilon}).
\end{align*}
Now, using the second-order Taylor expansion for $j=1,2,\ldots,[\omega(x_k)]$, we have 
\allowdisplaybreaks
\begin{align*}
  ~&~ F(x_k+p_k,\xi_{\beta_j^k}) \\ 
 \leq ~&~ F(x_k,\xi_{\beta_j^k})+ \nabla F(x_k,\xi_{\beta_j^k})^\top p_k+ \frac{1}{2} p_k^\top B(x_k,\xi_{\beta_j^k}) p_k +\frac{\epsilon}{2} \big\lVert p_k \big\rVert^2\\
 \leq ~&~ F(x_k,\xi_{\beta_j^k})+\gamma \big( \nabla F(x_k,\xi_{\beta_j^k})^\top p_k+ \frac{1}{2} p_k^\top B(x_k,\xi_{\beta_j^k}) p_k\big)\\ ~&~ + (1-\gamma) \big(\nabla F(x_k,\xi_{\beta_j^k})^\top p_k+ \frac{1}{2} p_k^\top B(x_k,\xi_{\beta_j^k}) p_k\big)+ \frac{\epsilon}{2} \big\lVert p_k \big\rVert^2\\
 \leq ~&~ F(x_k,\xi_{\beta_j^k})+\gamma \big( \nabla F(x_k,\xi_{\beta_j^k})^\top p_k+ \frac{1}{2} p_k^\top B(x_k,\xi_{\beta_j^k}) p_k\big)\\ ~&~ + (1-\gamma)\left(\sum_{j=1}^{[\omega(x_k)]}\lambda_j\nabla F(x_k,\xi_{\beta_j^k})^\top p_k+\frac{1}{2}p_k^\top\left(\sum_{j=1}^{[\omega(x_k)]}\lambda_j B(x_k,\xi_{\beta_j^k})\right)p_k\right)+\frac{\epsilon}{2} \big\lVert p_k \big\rVert^2\\
=~&~F(x_k,\xi_{\beta_j^k})+\gamma \big( \nabla F(x_k,\xi_{\beta_j^k})^\top p_k +\frac{1}{2} p_k^\top B(x_k,\xi_{\beta_j^k}) p_k\big)\\ ~&~ + (1-\gamma)\biggl(-p_k^\top\biggl(\sum_{j=1}^{[\omega(x_k)]}\lambda_j B(x_k,\xi_{\beta_j^k})\biggl)p_k\\
   ~&~ +\frac{1}{2}p_k^\top\biggl(\sum_{j=1}^{[\omega(x_k)]}\lambda_j B(x_k,\xi_{\beta_j^k})\biggl)p_k\biggl)+\frac{\epsilon}{2} \big\lVert p_k \big\rVert^2 \text{ using }\eqref{4.18}\\
\leq ~&~F(x_k,\xi_{\beta_j^k})+\gamma \big(\nabla F(x_k,\xi_{\beta_j^k})^\top p_k +\frac{1}{2} p_k^\top B(x_k,\xi_{\beta_j^k}) p_k\big)-\frac{a(1-\gamma)}{2}\big\lVert p_k \big\rVert^2 +\frac{\epsilon}{2} \big\lVert p_k \big\rVert^2 \\
\leq ~&~F(x_k,\xi_{\beta_j^k})+\gamma \big(\nabla F(x_k,\xi_{\beta_j^k})^\top p_k +\frac{1}{2} p_k^\top B(x_k,\xi_{\beta_j^k}) p_k\big)+\frac{\epsilon-a(1-\gamma)}{2}\big\lVert p_k \big\rVert^2,
\end{align*}
where from condition (iv), we conclude that 
\begin{align*}
\epsilon\leq a(1-\gamma) \implies \frac{\epsilon- a(1-\gamma)}{2}\leq 0,
\end{align*}
and $\tau_k =1$ holds for $k\geq k_{\epsilon}$. Now, for $k\geq k_{\epsilon}$, we have 
\begin{align*}
    \big\lVert x_{k+1} -x_{k+2} \big\rVert = \big\lVert p_{k+1} \big\rVert\leq\frac{\epsilon}{a}\big\lVert p_k \big\rVert =\frac{\epsilon}{a}\big\lVert x_k -x_{k+1} \big\rVert, \text{ using } \eqref{4.20},
\end{align*}
and for any $k\geq 1$ and $j\geq 1$, we obtain
\begin{align*}
\big\lVert x_{k+j} -x_{k+j+1} \big\rVert &\leq \frac{\epsilon}{a}\big\lVert x_{k+j-1} -x_{k+j} \big\rVert 
 \leq (\frac{\epsilon}{a})^2\big\lVert x_{k+j-2} -x_{k+j-1} \big\rVert 
 \leq \cdots \leq (\frac{\epsilon}{a})^j\big\lVert x_k -x_{k+1} \big\rVert.
\end{align*}
Now, to prove superlinear convergence, take $t\in(0,1)$ and define 
\begin{align*}
\bar{\epsilon}:=\text{min}\big\{a(1-\gamma),\frac{t}{1+2t}a\big\}.
\end{align*}
If we take $\epsilon<\bar{\epsilon}$ and $k\geq k_{\epsilon}$, then by convergence of sequence $\{x_k\}$ to $\bar{x}$, we have 
\begin{align}\label{T-4.8}
\big\lVert \bar{x}-x_{k+1} \big\rVert \leq \sum_{j=1}^{\infty} \big\lVert x_{k+j}-x_{k+j+1} \big\rVert  \leq \sum_{j=1}^{\infty} \big(\frac{t}{1+2t}\big)^j\big\lVert x_k-x_{k+1} \big\rVert = \frac{t}{1+t}\big\lVert x_k-x_{k+1}\big\rVert.
\end{align}
Therefore, we obtain
\begin{align}\label{T-4.9}
\big\lVert \bar{x}-x_k \big\rVert\geq \big\lVert x_k-x_{k+1} \big\rVert-\big\lVert x_{k+1}-\bar{x} \big\rVert \geq \frac{1}{1+t}\big\lVert x_k-x_{k+1} \big\rVert.
\end{align}
Combining the inequalities \eqref{T-4.8} and \eqref{T-4.9}, we can conclude that if $\epsilon<\bar{\epsilon}$ and $k\geq k_{\epsilon}$, then 
\begin{align*}
\frac{\big\lVert \bar{x}-x_{k+1} \big\rVert}{\big\lVert \bar{x}-x_k \big\rVert} \leq t.
\end{align*}
As $t$ lies in the interval $(0,1)$, we conclude that sequence $\{x_k\}$ converges superlinear to $\bar{x}$.
\end{proof}

\section{Numerical Illustration}\label{section:5}

In this section, we show the performance of the proposed quasi-Newton Algorithm \ref{QNM_algoritm} on several numerical examples. The testing of Algorithm \ref{QNM_algoritm} is conducted in MATLAB R2023b. The software is installed on an IOS system with an 8-core CPU and 8 GB RAM. For the numerical execution of the Algorithm \ref{QNM_algoritm}, we have selected the following parameter values:
\begin{itemize}
\item In our tests, the cone $K$ is chosen as the standard ordering cone, specifically $K=\mathbb{R}^{2}_{+}$ for all cases except in Example \ref{Example4}. Additionally, we set the parameter $e = (1,1,\ldots,1)^\top\in \text{int}(K)$ for the scalarizing function $\Theta_e$.
\item The parameter $\gamma$ in Step $4$ of the line search in Algorithm \ref{QNM_algoritm} was set to $\gamma = 0.1$.
\item For the stopping condition, we choose $\lVert p_k\rVert \leq 0.0001$, or a maximum number of 1000 iterations are reached.
\item To find the set \text{Max}$(F_U(x_k), K)$ with respect to upper set less relation at the $k$-th iterations in Step 1 of Algorithm \ref{QNM_algoritm}, we implemented the common method of pair-wise comparison the elements in $F_U(x_k)$. 
\item In the $k$-th iteration of Step 2 in Algorithm \ref{QNM_algoritm}, for each $\beta\in P_k$, we determine a minimizer $p$ of the strictly convex problem $\underset{p\in\mathbb{R}^n}{\text{min}}\varphi_x(\beta,p)$. Next, we identify 
\begin{align*}
    (\beta_k,p_k) = \underset{\beta\in P_k}{\text{argmin }} \varphi_{x_k}(\beta,p)
\end{align*}
with the help of the inbuilt function $fmincon$ in MATLAB.
\item We take some test problems from the literature, while some are freshly introduced in this paper. For each problem, we generate 100 random initial points and make a three-column table. The following values have been collected for each of the examples from Examples \ref{Example4} to \ref{Example5}: 
\begin{enumerate}
\item \textbf{Initial Points}: The value represents the first column in the table, indicating the number of random initial points used in applying Algorithm \ref{QNM_algoritm}. 
\item \textbf{Iterations}: The second column with a 6-tuple (Min, Max, Mean, Median, Mode, SD) indicates the minimum, maximum, mean, median, mode, and standard deviation of the number of iterations in which the stopping condition is reached.
\item \textbf{CPU Time}: The third column comprises another 6-tuple (Min, Max, Mean, Median, Mode, SD) that shows the minimum, maximum, mean, median, least integer greater or equal to mode, and standard
deviation of the CPU time (in seconds) taken by Algorithm \ref{QNM_algoritm} in reaching the stopping condition.
\end{enumerate}
\end{itemize}
Moreover, numerical values are displayed with precision up to four decimal places to enhance clarity. In each problem, the values of $F$ at each iteration are graphically depicted using different colors: the initial points are presented in black and the final points in red. The intermediate points are represented in blue.  

Additionally, we evaluate the performance of the proposed quasi-Newton method (abbreviated as QNM) algorithm (Algorithm \ref{QNM_algoritm}) by comparing it with the existing Newton method (abbreviated as NM) for each of the four Examples \ref{Example4} to \ref{Example5}.

\begin{example}\label{Example4}
Let the uncertainty set be $U:= \{0.1, 0.2, 0.3, 0.4\}$. Consider the UVOP  with the bi-objective function $F : \mathbb{R} \times U \to \mathbb{R}^2$ defined as
\begin{align*}
F(x,\xi) := \mqty(2x^2+e^{\frac{x}{10}}+\frac{(10\xi-3)}{2}\\ 5x\text{cos}(x)+\frac{(-10\xi+3)}{2}\text{sin}^2(x)).
\end{align*}
The ordering cone $K$ is given by 
\begin{align*}
K:=\{(z_1,z_2)^\top \in\mathbb{R}^2 : 3z_1-z_2\geq 0, -z_1+3z_2\geq 0\}.
\end{align*}
 In Figure \ref{figrue4}, the outcome of Algorithm \ref{QNM_algoritm} is illustrated for a randomly selected starting point within the set $[-4.7,4.7]$. It can be seen that the points depicted with red color are optimal points of $F_U$ as the set $(1.5680, 0.7815)^\top-K$ does not contain any element of $F_U$ other than $(1.5680, 0.7815)^\top$ for all $x\in[-4.7, 4.7]$.

The performance of Algorithm \ref{QNM_algoritm} and comparing it with the Newton method \cite{17} for Example \ref{Example4} is shown in Table \ref{Table4}. As usual, the Newton method performs better than the quasi-Newton method since for each $\xi$ in $U$, the function $F(\cdot,\xi)$ is strongly convex, and the Newton method has a faster rate (quadratic, see \cite{17}) of convergence than the superlinear rate of the proposed quasi-Newton method.

\begin{table}[H]  
\caption{  Performance of Algorithm \ref{QNM_algoritm} on Example \ref{Example4}}\label{Table4} 
\centering 
\begin{adjustbox}{max width=\textwidth}
\begin{tabular}{c c c c} 
\thickhline \thickhline
 Number of & Algorithm & Iterations  &{CPU time}\\
 successive points & & (Min, Max, Mean, Median, Mode, SD) & (Min, Max, Mean, Median, Mode, SD)
\\ [0.5ex]  
\hline  
 100 & QNM & (0, 4, 0.9800, 0, 0, 1.2551) &(0.1887, 1.4882, 0.5055, 0.2645, 0.1887, 0.3518) \\[0.5ex] 
\hline
100 & NM & (0, 3, 0.7300, 1, 0, 0.8391) & (0.2311, 1.9947, 0.6042, 0.6428, 0.2311, 0.4036)\\[0.5ex]
\thickhline 
\end{tabular}
\end{adjustbox}
\end{table}

\begin{figure} 
\includegraphics[width=0.6\textwidth]{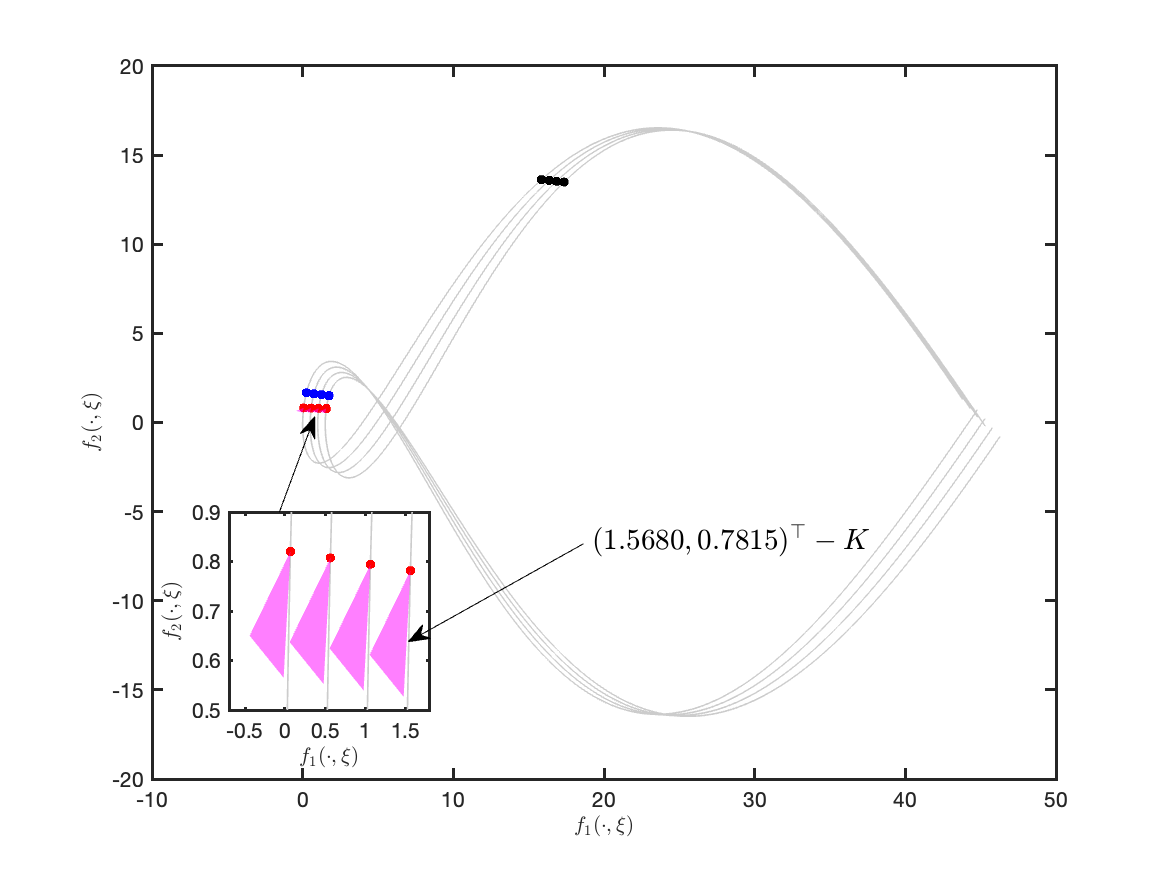} 
\caption{The value of $F_U$ at each iteration generated by Algorithm \ref{QNM_algoritm} for the initial point $x_0 = -2.8372$ for Example \ref{Example4}} \label{figrue4}
\end{figure}
\end{example}

\begin{example}\label{Example1}
Let the uncertainty set be $U:= \{0.1, 0.2, \ldots, 4.5\}$. Consider the UMOP  with the bi-objective function $F : \mathbb{R}^2 \times U \to \mathbb{R}^2$ defined as 
\begin{align*}
 F(x,\xi) := \mqty(x_1^2+x_2^2+0.5 \text{ sin}\left(\frac{2\pi(10\xi-1)}{60}\right)\text{cos}\left(\frac{2\pi(10\xi-1)}{60}\right)^2+2 e^{(x_1+x_2)} \\2x_1^2+2x_2^2+0.5 \text{ cos }\left(\frac{2\pi(10\xi-1)}{60}\right)^2).
\end{align*}
In Figure \ref{figure1}, the iterates generated by Algorithm \ref{QNM_algoritm} for different initial points taken from the set $[-1.5,1.5]\times [-1, 0.5]$ are given. The sequence of iterates $\{x_k\}$ and the corresponding $\{F(x_k)\}$ generated by Algorithm \ref{QNM_algoritm} for a selected initial point $x_0 = (1.5, 0.5)^\top$ are given in Figure \ref{example 1a} and Figure \ref{example 1b}, respectively.

The performance of Algorithm \ref{QNM_algoritm} and comparing it with the Newton method \cite{17}  for Example \ref{Example1} is shown in Table \ref{Table1}. In this example also, the Newton method performs better than Algorithm \ref{QNM_algoritm} because $F$ is strongly convex.

\begin{table}[H]  
\caption{ Algorithm \ref{QNM_algoritm} Performance of Algorithm \ref{QNM_algoritm} on Example \ref{Example1}}\label{Table1} 
\centering 
\begin{adjustbox}{max width=\textwidth}
\begin{tabular}{c c c c} 
\thickhline \thickhline
Number of & Algorithm & Iterations  &{CPU time}\\
successive points & & (Min, Max, Mean, Median, Mode, SD) & (Min, Max, Mean, Median, Mode, SD)
\\ [0.5ex]  
\hline  
 100 & QNM & (0, 11, 4.0220, 4, 3, 2.4266) &(0.0213, 32.3012, 14.2624, 13.2682, 2.3446, 6.8122) \\[0.5ex]
\hline 
100  & NM & (0, 11, 1.4700, 1, 1, 1.5005) &  (0.0191, 12.1958, 1.2327, 1.0054, 0.0191, 1.3375)  \\[0.5ex]  
\thickhline 
\end{tabular}
\end{adjustbox}
\end{table}

\begin{figure} 
     \centering
    \subfloat[The value of $F_U$ at each iteration generated by Algorithm \ref{QNM_algoritm} for the initial point $x_0 =(1.5,0.5)$ for Example \ref{Example1} \label{example 1a}]
{\includegraphics[width=0.45\textwidth]{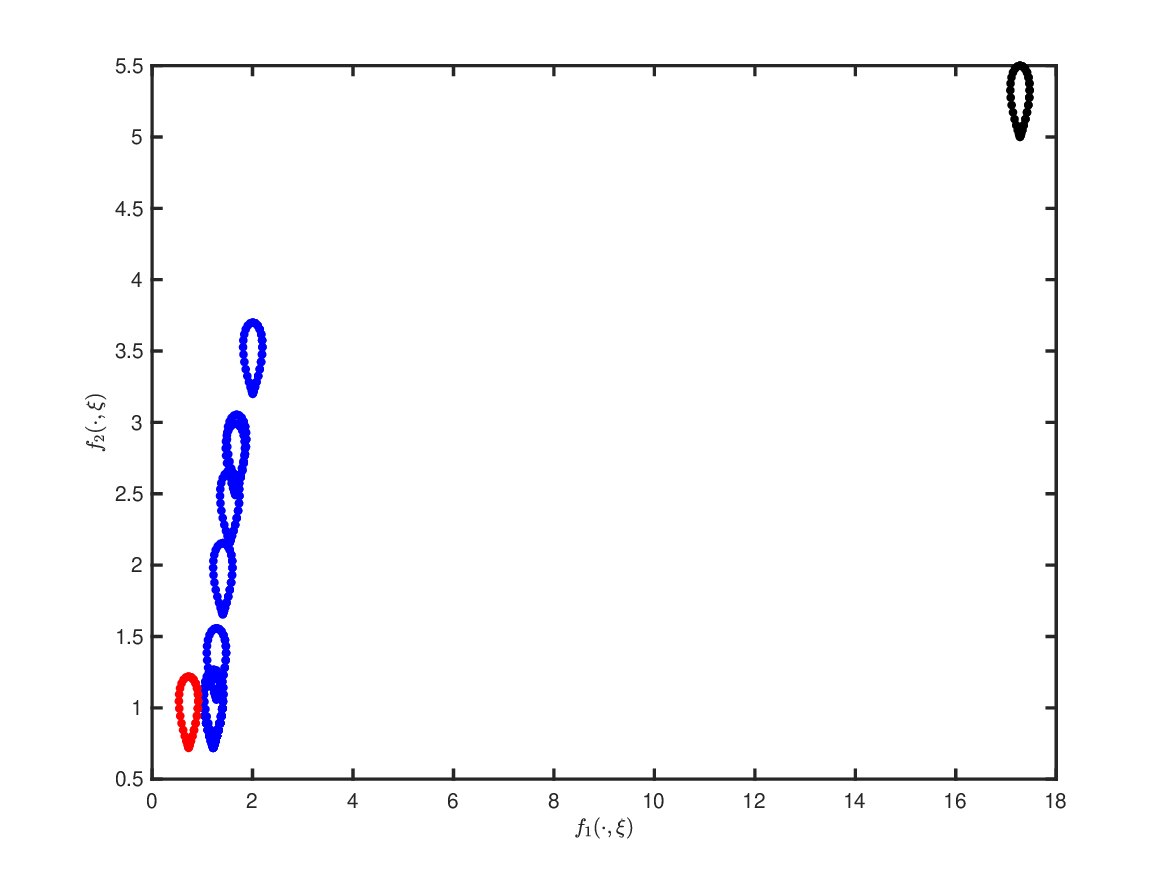}}
\hspace{0.25cm}
\subfloat[The value of $x_k$ at each iteration generated by Algorithm \ref{QNM_algoritm} for the initial points $x_0=(1.5,0.5)$ for Example \ref{Example1}\label{example 1b}]
{\includegraphics[width=0.45\textwidth]{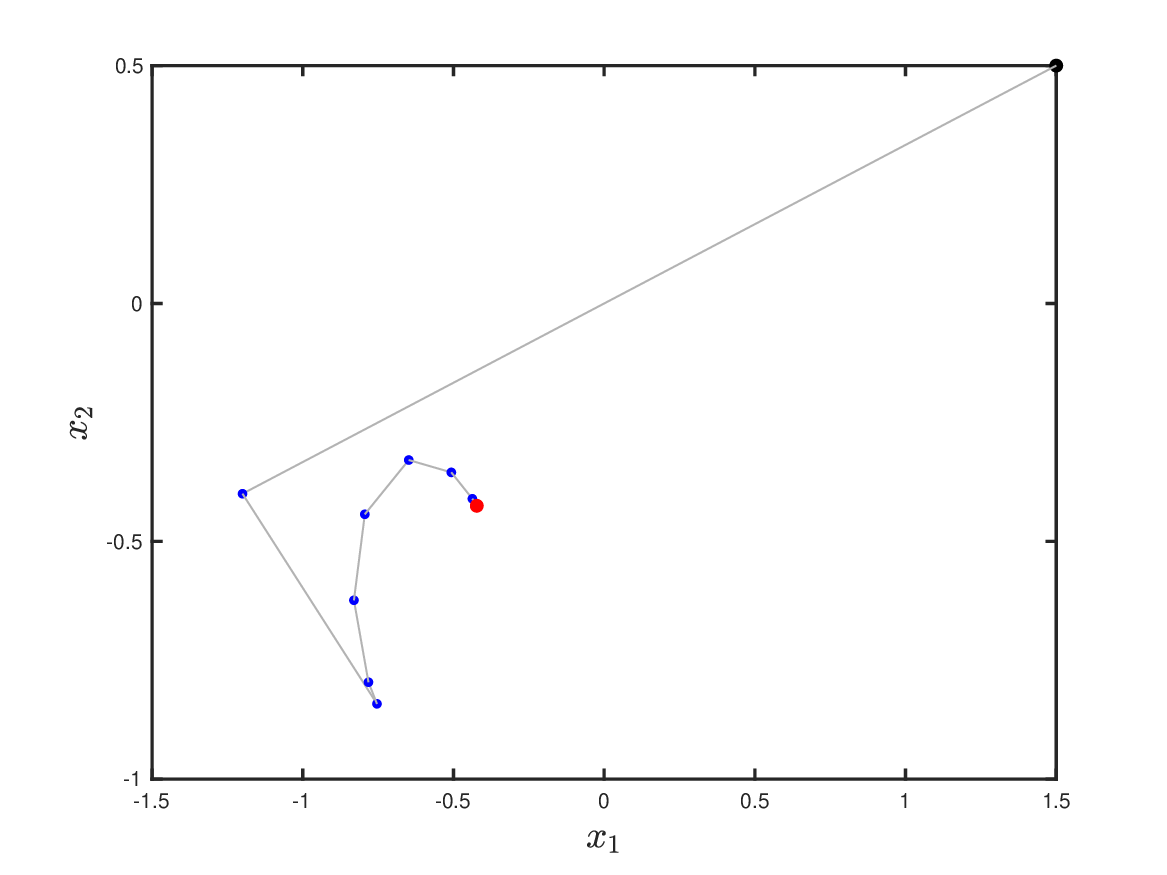}}
\caption{Output of Algorithm \ref{QNM_algoritm} for Example \ref{Example1}} \label{figure1}
\end{figure}
\end{example}

\begin{example}\label{Example3}
Let the uncertainty set be $U:= \{0.1, 0.2, \ldots, 5.0\}$. Consider the UMOP  with the tri-objective function $F : \mathbb{R}^2 \times U \to \mathbb{R}^3$ defined as
\begin{align*}
F(x,\xi) := \mqty(x_1^2+(0.5)\text{ sin }\left(\frac{2\pi(10\xi-1)}{50}\right)-0.1\xi x_1 \\2x_1^2+(0.5)\text{ cos }\left(\frac{2\pi(10\xi-1)}{50}\right)+0.2\xi x_2\\x_1^2+x_2^2+10\xi).
\end{align*}
The sequence of iterates $\{x_k\}$ and the corresponding $\{F_U(x_k)\}$ generated by Algorithm \ref{QNM_algoritm} for a selected initial point $x_0 = (1, -1)^\top$ are given in Figure \ref{example 3a} and Figure \ref{example 3b}, respectively.

Note that the objective function of this example does not fulfill the prerequisites of the Newton method given in \cite{17} because Hessian matrices are not positive definite at any given point. Therefore, in Table \ref{Table3}, we provide only the performance of the proposed quasi-Newton quasi-Newton method.

\begin{table}   
\caption{ Performance of Algorithm \ref{QNM_algoritm} on Example \ref{Example3}}\label{Table3} 
\centering 
\begin{adjustbox}{max width=\textwidth}
\begin{tabular}{c c c c} 
\thickhline \thickhline
 Number of & Algorithm & Iterations  &{CPU time}\\
 successive points & & (Min, Max, Mean, Median, Mode, SD) & (Min, Max, Mean, Median, Mode, SD)
\\ [0.5ex]  
\hline  
 100 & QNM & (0, 5, 1.6383, 1, 1, 1.5300) &(6.3462, 54.3345, 21.1840, 15.9259, 6.3462, 12.8506) \\[0.5ex]
\thickhline 
\end{tabular}
\end{adjustbox}
\end{table}

\begin{figure} 
     \centering
    \subfloat[The value of $F_U$ at each iteration generated by Algorithm \ref{QNM_algoritm} for the initial point $x_0 =(1, -1)$ for Example \ref{Example3} \label{example 3a}]
{\includegraphics[width=0.5\textwidth]{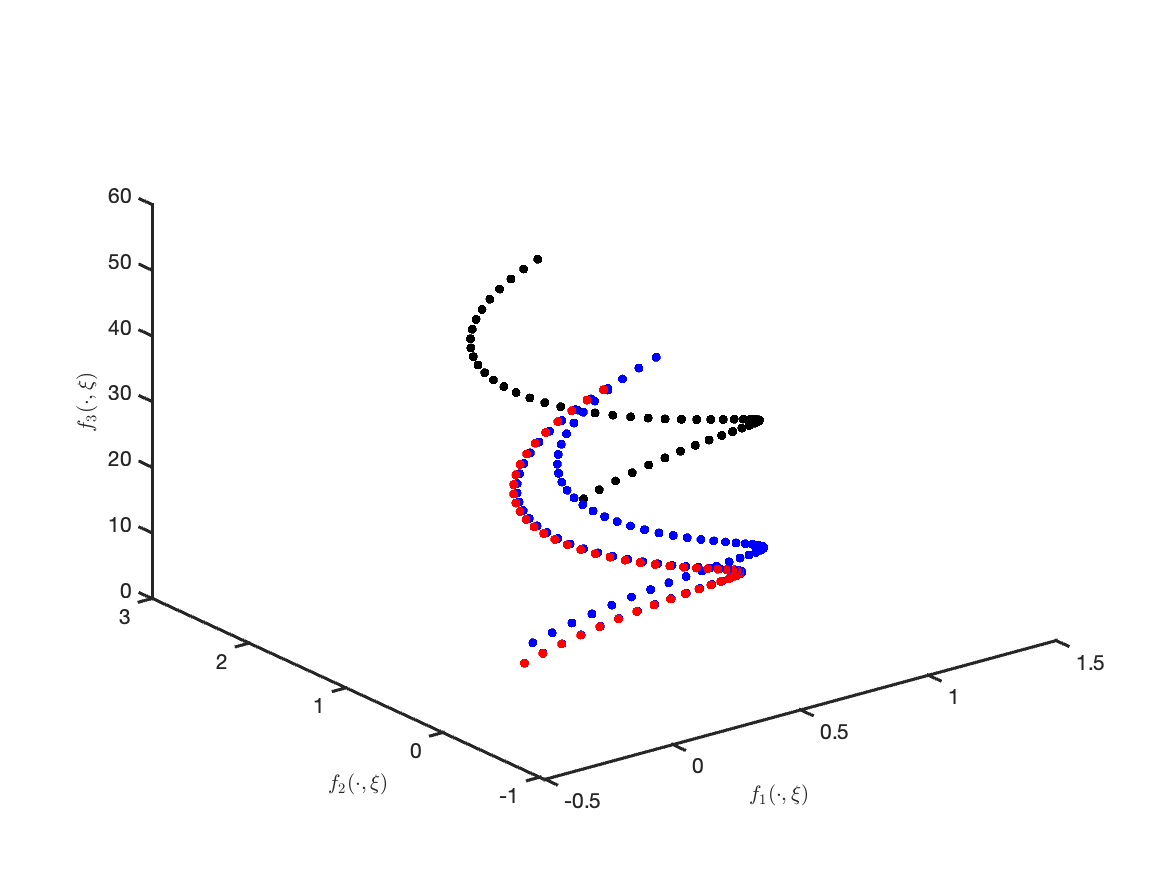}}
\hspace{0.3cm}
\subfloat[The value of $x_k$ at each iteration generated by Algorithm \ref{QNM_algoritm} for the initial points $x_0=(1, -1)$ for Example \ref{Example3} \label{example 3b}]
{\includegraphics[width=0.4\textwidth]{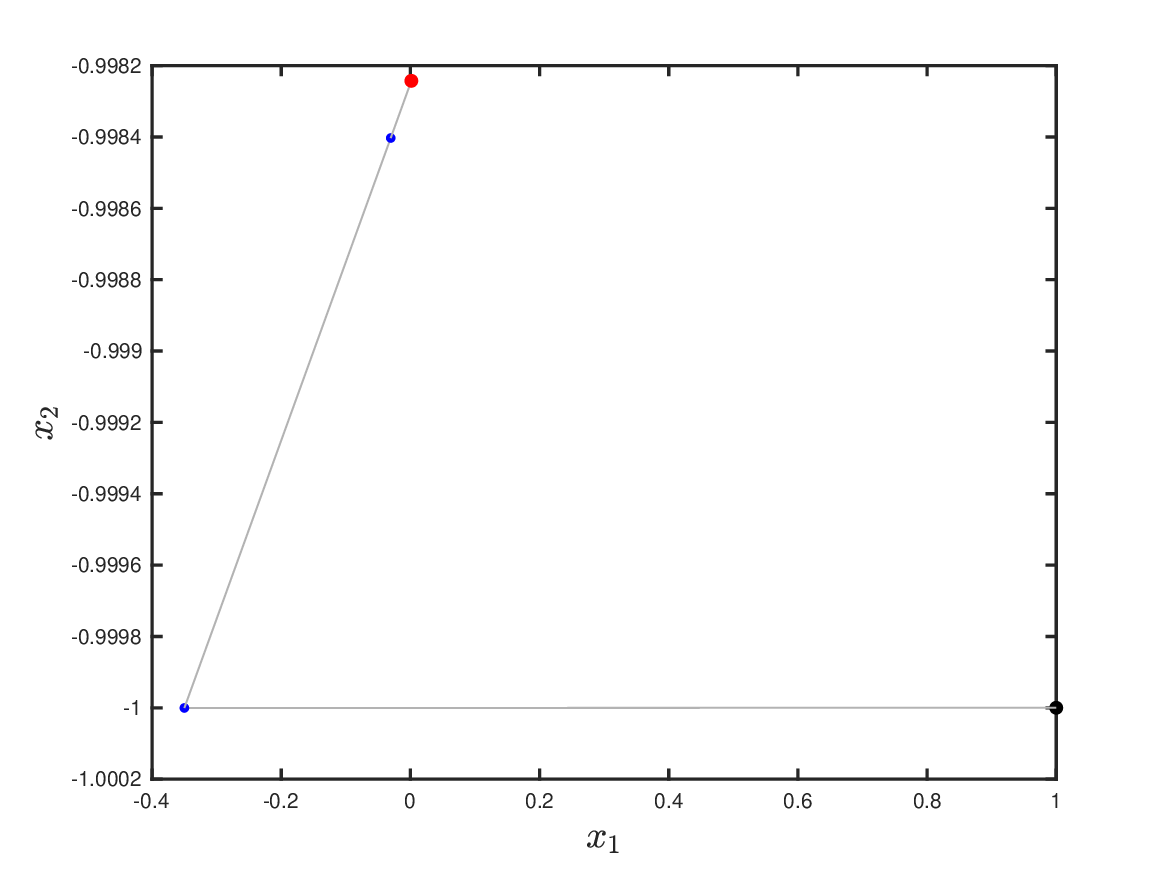}}
    \caption{Output of Algorithm \ref{QNM_algoritm} for Example \ref{Example3}}\label{figrue3}
\end{figure}
\end{example}

\begin{example} \cite{17} \label{Example5}
Consider the UMOP with the tri-objective function $F : \mathbb{R}^2 \times U \to \mathbb{R}^3$ defined as
\begin{align*}
F(x,\xi) := \frac{1}{2}\mqty(\lVert x-l_1-\xi\rVert^2\\ \lVert x-l_2-\xi\rVert^2\\\lVert x-l_3-\xi\rVert^2), 
\end{align*}
where $l_1: = \mqty(0\\0), l_2 := \mqty(8\\0)$ and $l_3 := \mqty(0\\8).$ We consider a uniform partition set of 10 points of the interval $[-1,1]$ given by 
\begin{align*}
  V := \left\{-1,-1+\frac{1}{s}, -1+\frac{2}{s},\ldots, -1+\frac{(2s-1)}{s},1\right\} ~\text{with}~ s = 4.5. 
\end{align*}
A scenario $\xi := (\xi_1, \xi_2)$ is a member of the uncertainty set $U:=V\times V$. In Figure \ref{figrue5}, $100$ initial points were generated within the square $[-50,50]\times [-50,50]$. The grey points illustrate the set $(l_1+\xi)\cup(l_2+\xi)\cup(l_3+\xi)$ and the positions of $l_1,l_2,l_3$ are shown in blue colour. The value of initial and final points generated by the Algorithm \ref{QNM_algoritm} are depicted by black and red, respectively, in Figure \ref{figrue5}. There are no intermediate points in Figure \ref{figrue5} because the Algorithm \ref{QNM_algoritm} has taken only one iteration to reach stationary points from these chosen initial points.

The performance of Algorithm \ref{QNM_algoritm} and comparing it with the Newton method \cite{17}  for Example \ref{Example5} is shown in Table \ref{Table5}. As usual, in this example, also the Newton method outperforms the quasi-Newton method since the function $F$ is strongly convex for each $\xi$, and the Newton Method has a quadratic rate of convergence.  

\begin{table} 
\caption{ Performance of Algorithm \ref{QNM_algoritm} on Example \ref{Example5}}\label{Table5} 
\centering 
\begin{adjustbox}{max width=\textwidth}
\begin{tabular}{c c c c} 
\thickhline \thickhline
 Number of & Algorithm & Iterations  &{CPU time}\\
 successive points & & (Min, Max, Mean, Median, Mode, SD) & (Min, Max, Mean, Median, Mode, SD)
\\ [0.5ex]  
\hline  
 100 & QNM & (0, 2, 1.0800, 1, 1, 0.3075) &(6.1637, 12.7476, 9.8998, 8.5954, 6.1637, 2.0854) \\[0.5ex] 
 \hline 
 100 & NM & (0, 2, 1.0782, 1, 1, 0.3105) & (2.3863, 12.7695, 6.4920, 4.2540, 2.3863, 3.9367)\\[0.5ex] 
\thickhline 
\end{tabular}
\end{adjustbox}
\end{table}

\begin{figure} 
\includegraphics[width=0.5\textwidth]{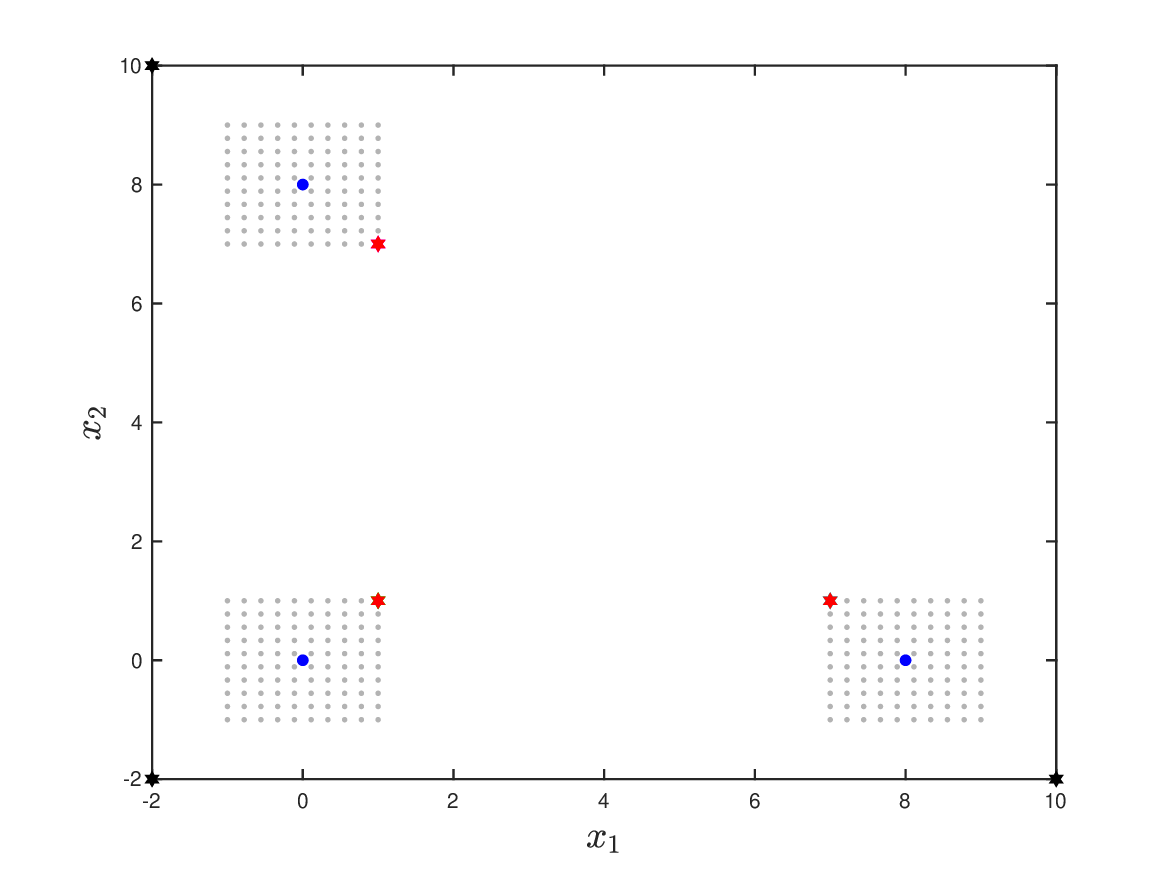} 
\caption{Output of Algorithm \ref{QNM_algoritm} for Example \ref{Example5}: The value of $x_k$ at each iteration for three different initial points: $\{(-2,-2)^\top,(-2,10)^\top,$ and $(10,-2)^\top\}$, as generated by Algorithm \ref{QNM_algoritm} in \ref{Example5}}\label{figrue5}
\end{figure}
\end{example}

Next, we generate a few test problems and show the performance of the proposed Algorithm \ref{QNM_algoritm} on those problems. These test problems are either newly introduced or generated from the commonly used test problems in the literature of multiobjective optimization. The list of the generated test problems is given in Table \ref{test_problem}. In the last column of the table, we mention the source of the associated multiobjective problem. For example, the problem P3 is generated from the bi-objective test problem BK1 from \cite{32}. In the conventional bi-objective test problem BK1, we have incorporated the uncertainty variables $\xi_1$ and $\xi_2$ to generate the uncertain bi-objective test problem P3; note that for $\xi_1 = 1$ and $\xi_2 = 1$, the problem P3 reduces to BK1 problem. Similarly, all other problems, except P2 and P7, are generated from the sources mentioned in the last column of Table \ref{test_problem}. Problems P2 and P7 are newly introduced problems for UMOP.

The performance of Algorithm \ref{QNM_algoritm} on the test problems from Table \ref{Table5} are provided in Table \ref{Table6_performance}. In Figure \ref{iterations_on_test_problems}, we have shown the generated 
robust weakly efficient point corresponding to a randomly chosen initial point for test problems taken from Table \ref{test_problem}. The black-colored points in each figure collectively present the value of $F_U(x_0)$ at the initial point $x_0$, blue-colored points present $F_U(x_k)$ for intermediate $x_k$'s, and the red-colored points collectively represent $F_U(\bar x)$, where $\{x_k\}$ is the sequence generated by Algorithm \ref{QNM_algoritm} corresponding to a given initial point $x_0$ and $x_k \to \bar x$. The point $\bar x$ is a robust weakly efficient point of the corresponding point.

\begin{table} 
    \caption{Test problems}\label{test_problem} 
\centering 
\resizebox{17cm}{!}{
\begin{tabular}{c c l c c c c c } 
\toprule
\multirow{2}{*}{Problem} & \multirow{2}{*}{($m$, $n$, $r$)} &  \multirow{2}{*}{Modified $f$}   & \multicolumn{2}{c}{Domain for $x$ } & \multicolumn{2}{c}{\text{Domain for scenarios}}  & \multirow{2}{*}{Reference for $f$}\\
 \cmidrule(rl){4-5} \cmidrule(rl){6-7}  
 &  &       & $lb^\top$ & $ub^\top$ & {$lb^\top$} &{$ub^\top$}   &
\\ [0.5ex]  
\midrule  
P1\label{P1} &(2,1,1) & $\begin{aligned}
     f_1(x, \xi) :=& x^2+2(\xi x+1)^2+2(\xi x-1)^2+x^3  \\ f_2(x, \xi) :=&(x-2\xi)^2+\xi^3 x+\xi x^3
 \end{aligned}$ &   $-1$ & $0.5$& $-2$ & $2$ & MOP1 \cite{17} \\[3ex]
\hline
P2\label{P2} &(2,1,1) & $\begin{aligned}
     f_1(x, \xi) :=& x+\tfrac{1}{140}\big(10+\exp\left\{\sin(\tfrac{2\pi\xi}{50})\right\}- \sin(\tfrac{4\pi\xi}{50})\big)\\&+\xi x\\ 
    f_2(x, \xi) :=& \cos(3x)+\tfrac{1}{1+\exp{(2x)}}+\frac{1}{140}\big(10+\exp\left\{\sin(\tfrac{2\pi\xi}{50})\right\}\\ &- \sin(\tfrac{4\pi\xi}{50})\big)\\    
 \end{aligned}$ &   $-1$ & $-0.7$& $1$ & $9$ & New1  \\[3ex]
\hline
P3\label{P3} &(2,2,2) & $\begin{aligned}
     f_1(x, \xi) :=& \xi_1 x_1^2 + \xi_2 x_2^2+\xi_2(x_2-5)^2\\&+(\xi_1x_1-2)^2+(\xi_2x_2-2)^2 \\ f_2(x, \xi) :=&\xi_1x_1^2+\xi_2x_2^2+(\xi_1x_1 - 5)^2 
 \end{aligned}$ &  $(-3, -3)$ & $(5, 5)$& $(-2,-2)$ & $(2,2)$ & BK1 \cite{32} \\[3ex]
 \hline
P4\label{P4} &(2,2,2) & $\begin{aligned}
     f_1(x, \xi) :=& x_1^2 + \xi_2 x_2^2+\xi_1x_1^2+8\xi_2x_2 \\ f_2(x, \xi) :=&\xi_1(x_1 + 1)^2 + x_2 ^2+\xi_1x_1^2+\xi_2x_2^2
 \end{aligned}$ &   $(-5,-5)$ & $(50, 50)$& $(0.8,1)$ & $(1,1.2)$ & LRS1 \cite{32} \\[3ex]
\hline
P5\label{P5} &(2,2,2) & $\begin{aligned}
     f_1(x, \xi) :=& (\xi_1 x_1^2 -1)^2 +\xi_2(x_1-x_2)^2\\ f_2(x, \xi) :=&(\xi_2 x_2 -3)^2 +\xi_1(x_1-x_2)^2
 \end{aligned}$ &   $(-1, -1)$ & $(5, 5)$& $(1,1)$ & $(2,2)$ & SP1 \cite{32} \\[3ex]
\hline
P6\label{P6} &(2,2,2) & $\begin{aligned}
     f_1(x, \xi) :=& \frac{\xi_1}{x_1^2 + \xi_2 x_2^2 + 1}+\xi_1^2x_1^2+(\xi_2x_2-1)^2\\ f_2(x, \xi) :=&\frac{\xi_1 x_1^2 + 3 x_2^2+1 }{\xi_2}+2\xi_1x_1+2\xi_2x_2
 \end{aligned}$ &   $(-3, -3)$ & $(3, 3)$& $(1,0.5)$ & $(1.5,1)$ & VU1 \cite{32} \\[3ex]
 \hline
 P7\label{P7} &(2,2,2) & $\begin{aligned}
     f_1(x, \xi) :=& \xi_2+x_1^2+(\xi_1x_1-2)^2+(\xi_2x_2-2)^2\\ f_2(x, \xi) :=& \xi_2+\xi_1^2+\xi_1x_1^2+\xi_2x_2^2+(\xi_1x_1-2)^3
 \end{aligned}$ &   $(-1, -1)$ & $(0, 0)$& $(1.2,0.6)$ & $(1.5,1)$ & New2  \\[3ex]
\hline
P8\label{P8} &(2,2,2) & $\begin{aligned}
     f_1(x, \xi) :=& -1.05\xi_1x_1^2-0.98\xi_2x_2^2+(\xi_1x_1-2)^2+(\xi_2x_2+2)^2  \\ f_2(x, \xi) :=&-0.99\xi_1(x_1-3)^2-1.03\xi_2(x_2-2.5)^2+5\xi_1x_1^2+\xi_2x_2^2
 \end{aligned}$ &   $(-3,-3)$ & $(5,5)$& $(-1.5,-1.5)$ & $(0,0)$ & Lovison1 \cite{40} \\[3ex]
\hline

P9\label{P9} &(2,3,3) & $\begin{aligned}
     f_1(x, \xi) :=& 2x_1^2+\frac{(\xi_1-3)}{2}+4x_2\frac{(\xi_2-3)}{2}  \\ f_2(x, \xi) :=&\frac{x_1^2}{4}\cos x_2 -\cos^3 x_3\frac{(\xi_1-3)}{2}+\xi_3
 \end{aligned}$ &   $(-100000,\ldots, -100000)$ & $(100000,\ldots, 100000)$& $(-2,\ldots,-2)$ & $(2,\ldots,2)$ & GKZ9 \cite{17} \\[3ex]
\hline
P10\label{P10} &(2,5,1) & $\begin{aligned}
     f_1(x, \xi) :=& x_1^2 + x_2^2+x_3^2 + x_4^2+x_5^2+\xi^2 \\ f_2(x, \xi) :=&(3+\xi) x_1 + 2x_2 -\tfrac{x_3}{3}+(x_4-x_5)^3+\xi x_1^2
 \end{aligned}$ &   $(-20,\ldots,-20)$ & $(20,\ldots,20)$& $0.008$ & $0.012$ & DD1 \cite{32} \\[3ex]
\hline
P11\label{11} &(2,20,20) & $\begin{aligned}
     f_1(x, \xi) :=& \frac{1}{n}\sum_{i=1}^{n}\xi^2_i x^2_i \\ f_2(x, \xi) :=&\frac{1}{n}\sum_{i=1}^{n}( x_i-2\xi_i)^2+\xi_i^2
 \end{aligned}$ &   $(-9,\ldots, -9)$ & $(-7,\ldots,-7)$& $(-2,\ldots,-2)$ & $(1,\ldots,1)$ & Jin1 \cite{33} \\[3ex]
\hline
P12\label{P12} &(3,2,2) & $\begin{aligned}
     f_1(x, \xi) :=& \frac{(\xi_1x_1-2)^2}{2}+\frac{(\xi_2x_2+1)^2}{13}+3\xi_1+x_1^2+x_2^2+\xi_1\xi_2 \\ f_2(x, \xi) :=&\frac{(\xi_1x_1+\xi_2x_2-3)^2}{36}+\frac{(-\xi_1x_1+\xi_2x_2+2)^2}{8}-17\xi_2+x_1+x_2+\xi_2\\
     f_3(x,\xi)=& \frac{(\xi_1x_1+2\xi_2x_2-1)^2}{175}+\frac{(-\xi_1x_1+2\xi_2x_2)^2}{17}-13\xi_1\xi_2+\xi_2x_1^2-9x_2
 \end{aligned}$ &   $(-400, -400)$ & $(400, 400)$& $(-2,-2)$ & $(1,1)$ & MOP7 \cite{32} \\[3ex]
\hline
P13\label{P13} &(3,2,2) & $\begin{aligned}
     f_1(x, \xi) :=& x_1 +(x_2-1.5) +\tfrac{1}{4}\sin\left(\frac{4\pi(\xi_1-1)}{7}\right)+\tfrac{1} {100}\xi_2+x_1^2\\ f_2(x, \xi) :=& 2(x_1-1)^2 +2x_2^2+\tfrac{1}{4}\cos\left(\tfrac{4\pi(\xi_2-1)}{7}\right)+\tfrac{2}{100}\xi_1+x_2^2\\
     f_3(x, \xi) :=& \xi_1x_1^2 +x_2^2+\xi_1\xi_2
 \end{aligned}$ &   $(1, 1)$ & $(1.5, 1.5)$& $(1,1)$ & $(1.5,2)$ & GKZ6 \cite{17} \\[3ex]
\hline
P14\label{P14} &(3,2,3) & $\begin{aligned}
     f_1(x, \xi) :=& \xi_1^2 x_1^2 +\xi_2(x_2-1)^2+x_3+\xi_1x_1^2+x_2^2\\ f_2(x, \xi) :=&x_1^2 +(\xi_1x_2-1)^2+x_1^4+\xi_2x_2^2+1\\
     f_3(x, \xi) :=& \xi_1(x_1 -1)^2 +\xi_2x_2^2+\xi_3+5\xi_1x_1+x_2+3
 \end{aligned}$ &   $(-2,-2)$ & $(2,2)$& $(1,0.5,-0.5)$ & $(1.5,1,2.5)$ & VFM1 \cite{32} \\[3ex]
\hline
P15\label{P15} &(3,2,3) & $\begin{aligned}
     f_1(x, \xi) :=& (\xi_1 x_1 -1)^2 +(\xi_2x_2-1)^2+\xi_1x_1^2+x_2^2\\ f_2(x, \xi) :=&(\xi_2 x_1 -1.5)^2 +(\xi_1x_2-1)^2+x_1^4+\xi_2x_2^2\\
     f_3(x, \xi) :=& \xi_2(x_1 -1)^2 +\xi_1(\xi_2x_2-1)^2+\xi_3+5\xi_1x_1+x_2
 \end{aligned}$ &   $(-4, -4)$ & $(4, 4)$& $(-2.5,-2.5,-0.5)$ & $(2.5,2.5,2.5)$ & MHHM2 \cite{23} \\[3ex]
\hline
P16\label{P16} &(3,10,10) & $\begin{aligned}
     f_1(x, \xi) :=& \sum_{i=1}^{n}\xi_i x_i +\xi_i^2\\
    f_2(x, \xi) :=& 1+9 \sum_{i=1}^{n}\xi_i^2 x_i\\
    f_3(x, \xi) :=& 1-\sqrt{\frac{f_1(x, \xi)}{f_2(x, \xi)}}
 \end{aligned}$ &   $(0.2,\ldots,0.2)$ & $(0.8,\ldots,0.8)$& $(-1,\ldots,-1)$ & $(1,\ldots,1)$ & ZDT1 \cite{35} \\[3ex]
\hline
P17\label{P17} &(3,10,10) & $\begin{aligned}
     f_1(x, \xi) :=& \sum_{i=1}^{n}\xi_i x_i \\
    f_2(x, \xi) :=& 5+10 \sum_{i=1}^{n}\xi_i x_i^2\\
    f_3(x, \xi) :=& 2-\left(\frac{f_1(x, \xi)}{f_2(x, \xi)}\right)^2
 \end{aligned}$ &   $(0.2,\ldots,0.2)$ & $(0.6,\ldots,0.6)$& $(-0.8,\ldots,-0.8)$ & $(0.7,\ldots,0.7)$ & ZDT2 \cite{35}  \\[3ex]
\hline
 P18\label{P18} &(3,10,10) & $\begin{aligned}
     f_1(x, \xi) :=& \tfrac{1}{n}\sum_{i=1}^{n}i(\xi_i x_i-i)^4 \\ f_2(x, \xi) :=&\exp\left(\sum_{i=1}^{n}\tfrac{x_i}{n}\right)\\
     f_3(x,\xi)=& \tfrac{1}{n(n+1)}\sum_{i=1}^{n}i(n-i+1)e^{-\xi_i x_i}
 \end{aligned}$ &   $(-2,\ldots, -2)$ & $(2,\ldots, 2)$& $(-1,\ldots,-1)$ & $(2,\ldots,2)$ & FDS \cite{34} \\[3ex]
\bottomrule
\end{tabular}}
\end{table}

\begin{table}[h!]
\centering
 \caption{Performance of Algorithm \ref{QNM_algoritm} for all problems given in Table \ref{test_problem} }\label{Table6_performance}
\resizebox{10cm}{!}{
\begin{tabular}{c c c c c c c c}
\toprule
\multicolumn{1}{c}{Problem} &  \multicolumn{1}{c}{Senarios} & \multicolumn{3}{c}{\text{Time}} & \multicolumn{3}{c}{\text{Iterations count}}\\
\cmidrule(rl){3-5} \cmidrule(rl){6-8}
 &  & \text{min} & \text{mean} & \text{max} & \text{min} & \text{mean} & \text{max}\\
\midrule 
P1 & 40 & 0.0629 & 0.5031 & 1.0282 & 0 & 4.76 & 9 \\
P2 & 35 & 1.5480 & 1.6323 & 1.9709 & 3 & 3.98 & 4 \\
P3 & 40 & 0.3044 & 1.2164 & 2.5341 & 0 & 1.60 & 7 \\
P4 & 40 & 0.5314 & 2.8990 & 62.3406 & 2 & 11.70 & 256 \\
P5 & 25 & 0.6717 & 5.1853 & 119.0370 & 1 & 8.6531 & 198 \\
P6 & 30 & 2.1577 & 2.2529 & 3.2651 & 2 & 2 & 2 \\
P7 & 20 & 0.8279 & 6.2513 & 73.1986 & 1 & 9 & 98 \\
P8 & 40 & 0.2294 & 12.8858 &  117.0456 & 0 & 16.10 & 149 \\
P9 & 40 & 1.609 & 99.1422 & 407.6247 & 7 & 106 & 911 \\
P10 & 25 & 14.2015 & 58.6909 & 144.2981 & 39 & 76.7234 & 155 \\
P11 & 20 & 23.8567 & 25.9025 & 30.3595 & 7 & 7.0392 & 8 \\
P12 & 30 & 0.4093 & 20.6827 & 202.9008 & 1 & 30.4583 & 313 \\
P13 & 30 & 2.9827 & 4.8822 & 10.6499 & 2 & 4.36 & 8 \\
P14 & 20 & 0.4212 & 2.3219 & 26.9343 & 1 & 8.58 & 136 \\
P15 & 40 & 0.1858 & 1.9971 & 28.8318 & 0 & 4.62 & 89 \\
P16 & 10 & 1.6223 & 4.0545 & 15.1023 & 0 & 2.53 & 12 \\
P17 & 10 & 0.6223 & 3.0545 & 16.1023 & 1 & 1.54 & 10 \\
P18 & 10 & 0.2532 & 23.6702 & 221.7030 & 0 & 30.5319 & 297 \\

\bottomrule
\end{tabular}} 
\end{table}

\begin{landscape}
\begin{figure}
\centering

\subfloat[Problem P1]
{\includegraphics[width=0.38\textwidth]{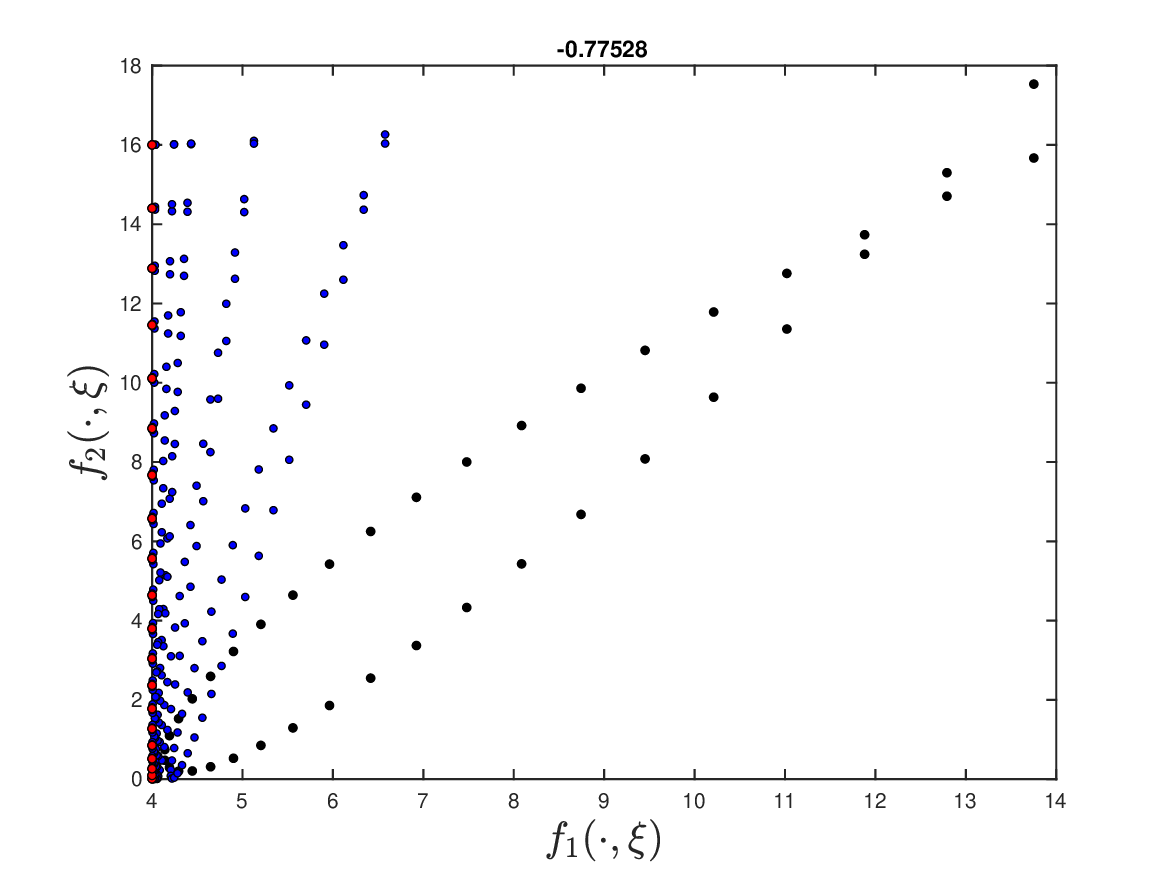}}
\hspace{0.01cm}
\subfloat[Problem P3]
{\includegraphics[width=0.38\textwidth]{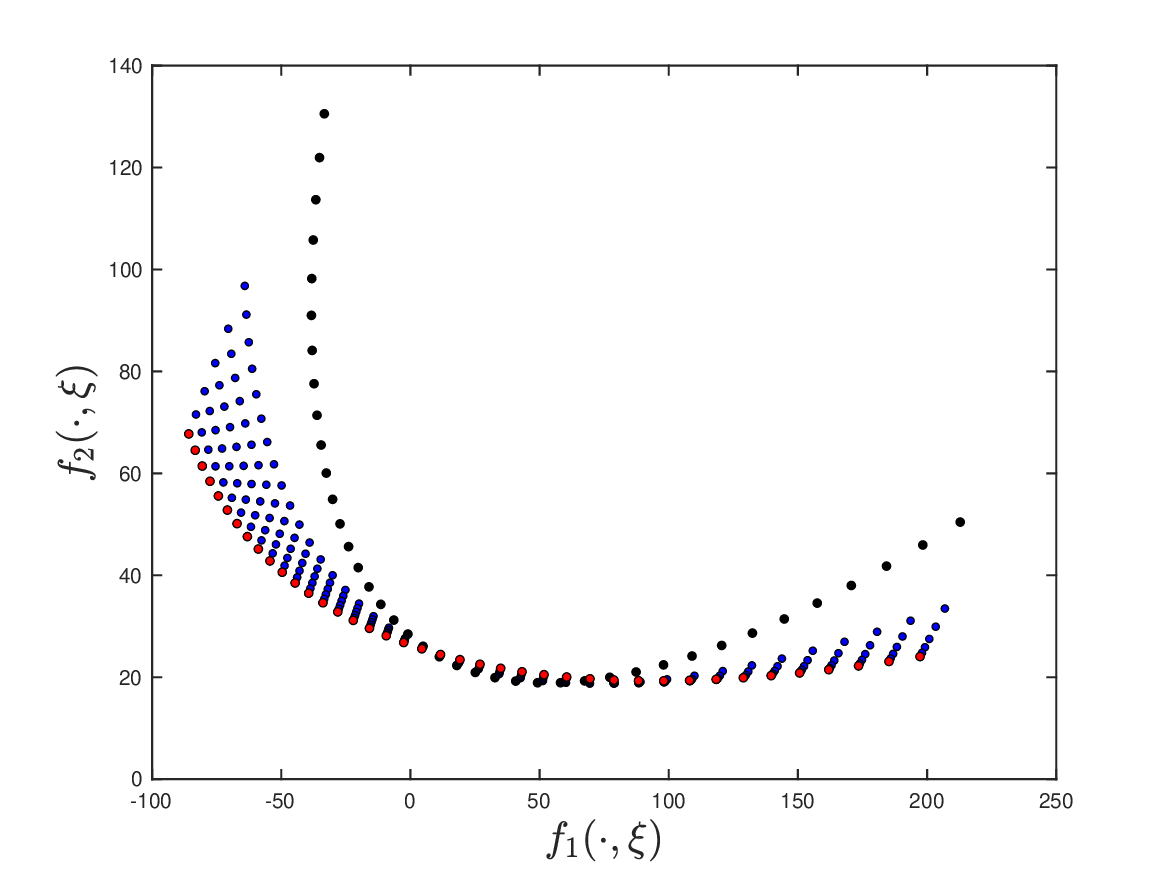}}
\hspace{0.05cm}
\subfloat[Problem P5]
{\includegraphics[width=0.38\textwidth]{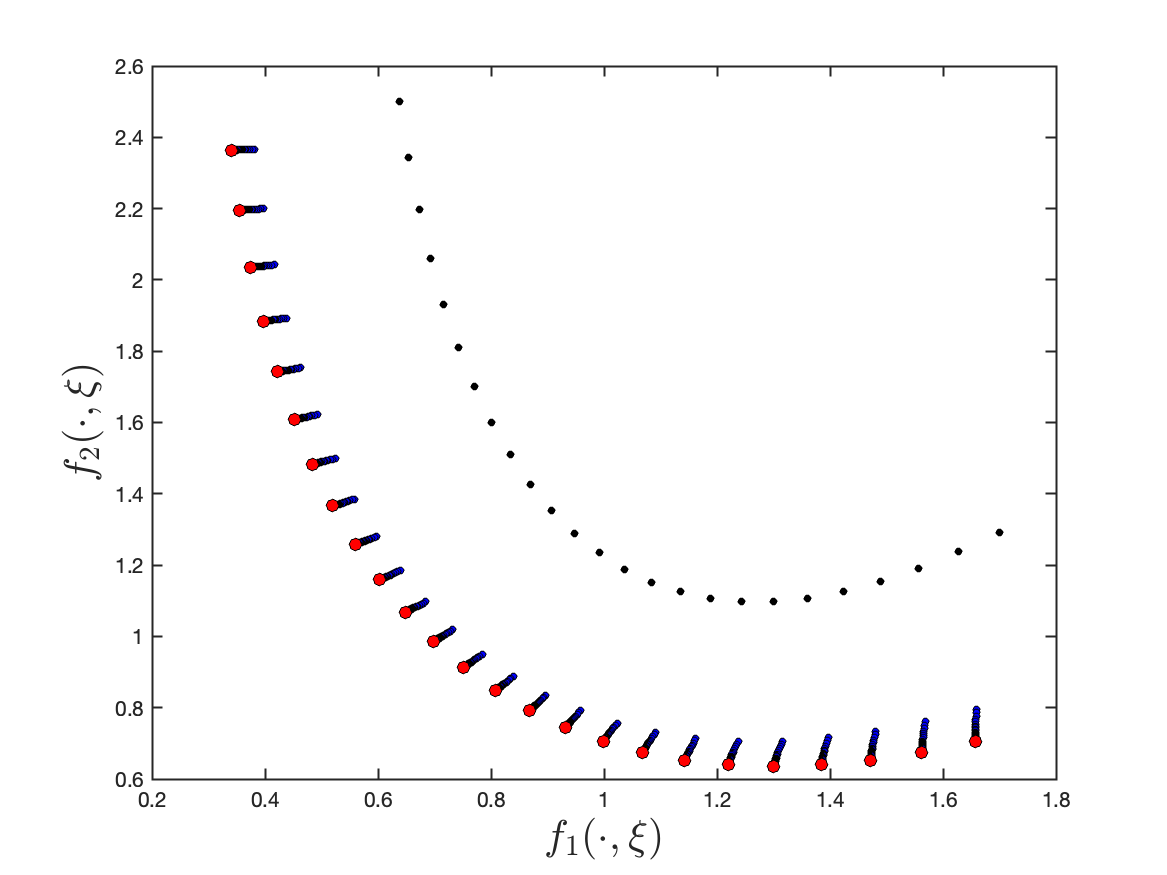}}\\
\subfloat[Problem P6]
{\includegraphics[width=0.38\textwidth]{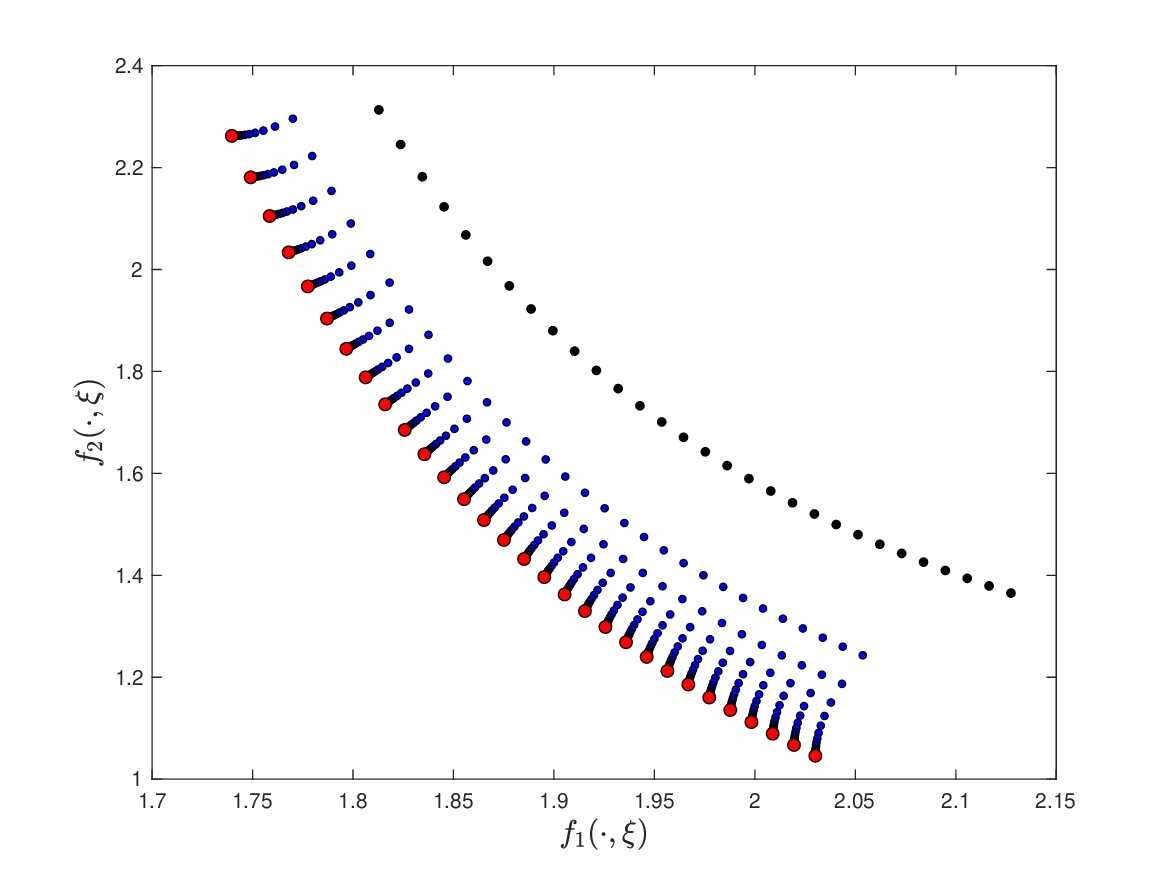}}
\hspace{0.05cm}
\subfloat[Problem P7]
{\includegraphics[width=0.38\textwidth]{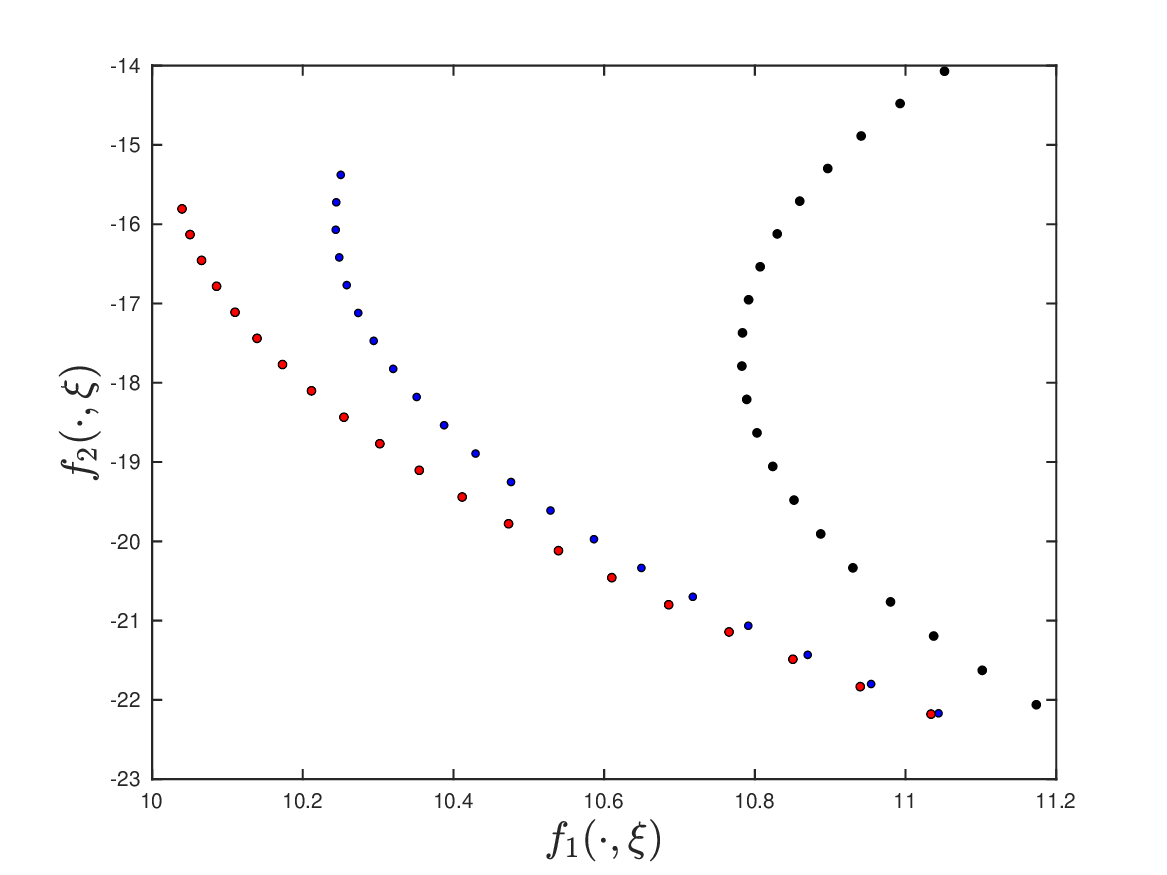}}  
\hspace{0.05cm}
\subfloat[Problem P8]
{\includegraphics[width=0.38\textwidth]{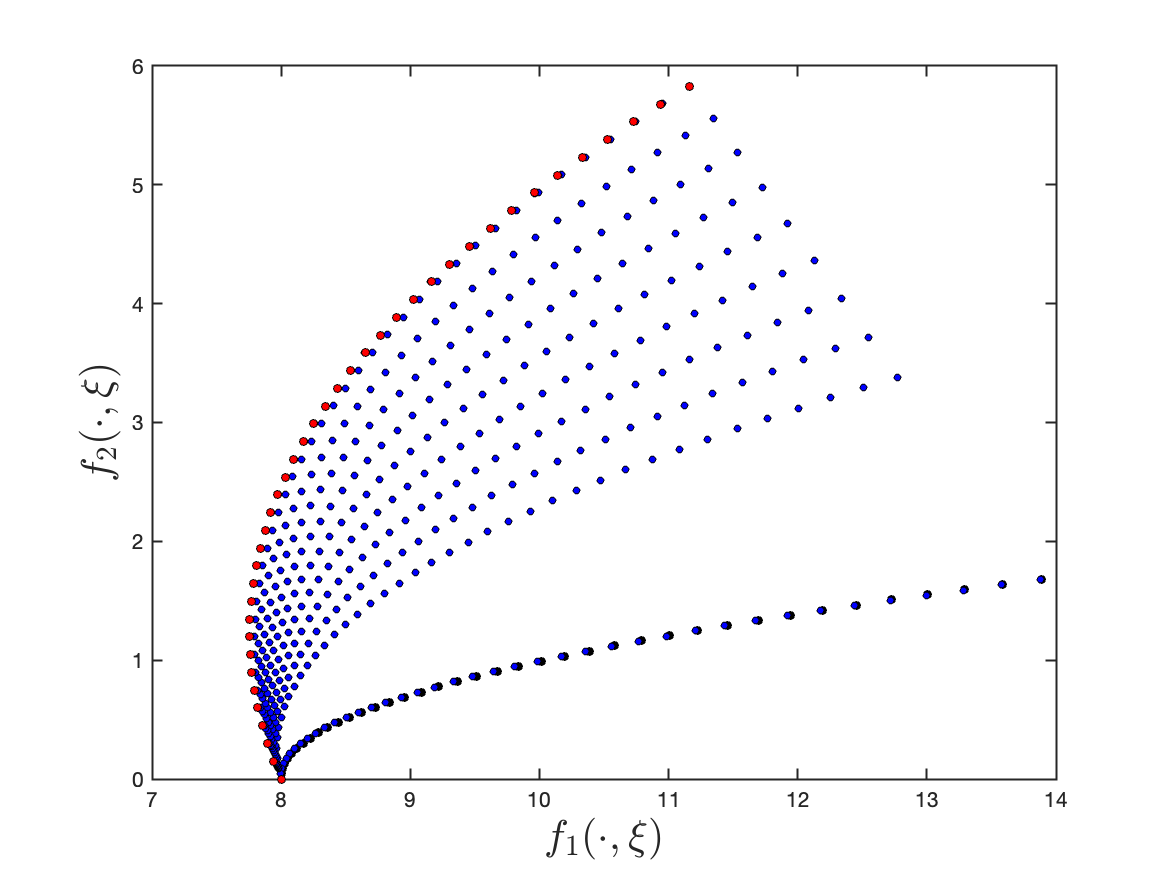}}\\
\subfloat[Problem P12]
{\includegraphics[width=0.38\textwidth]{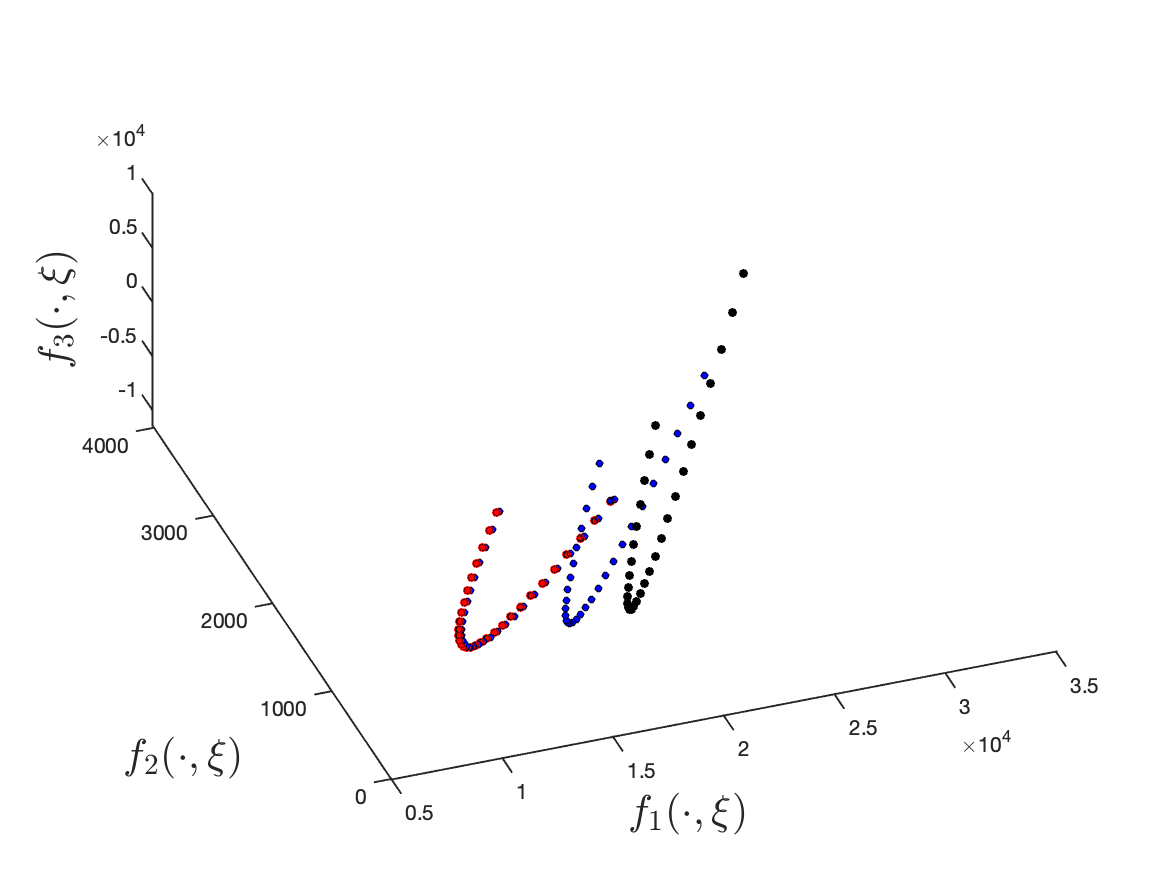}}
\hspace{0.05cm}
\subfloat[Problem P13]
{\includegraphics[width=0.38\textwidth]{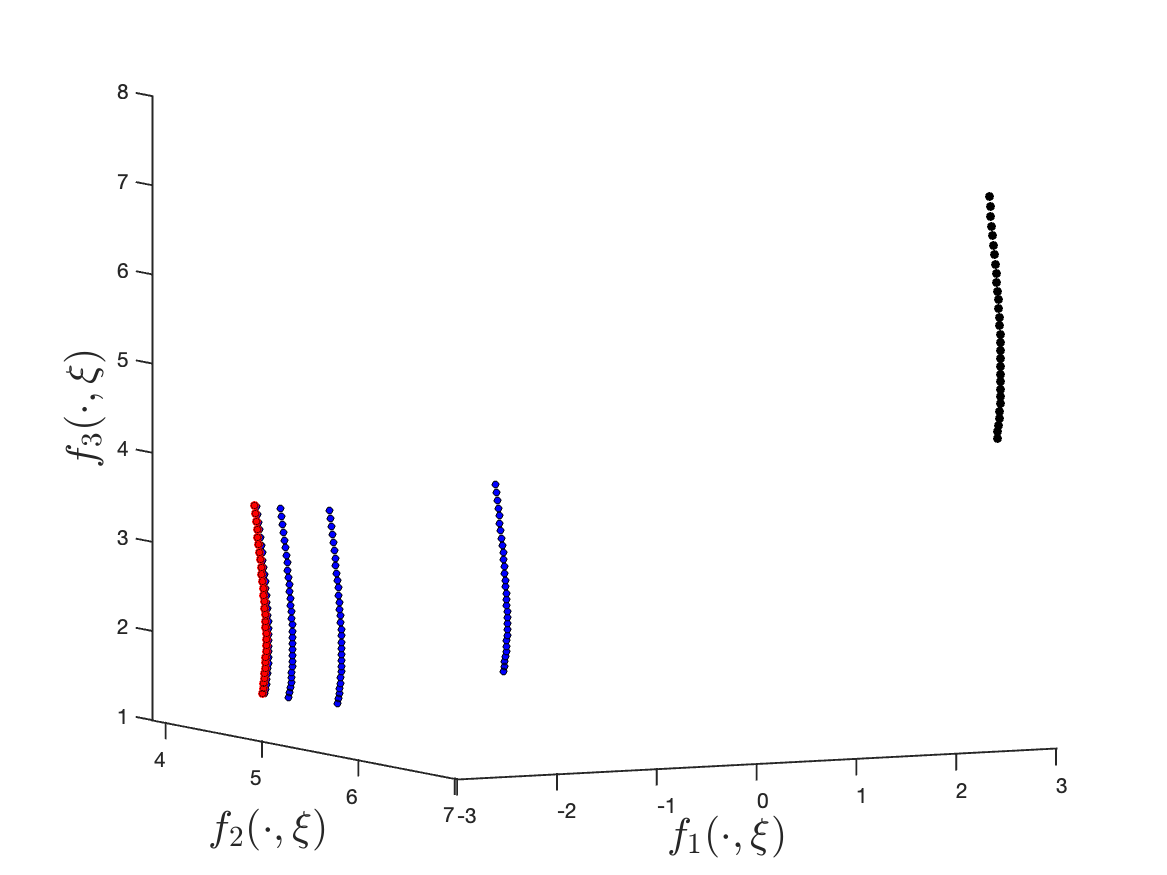}}
\hspace{0.05cm}
\subfloat[Problem P15]
{\includegraphics[width=0.38\textwidth]{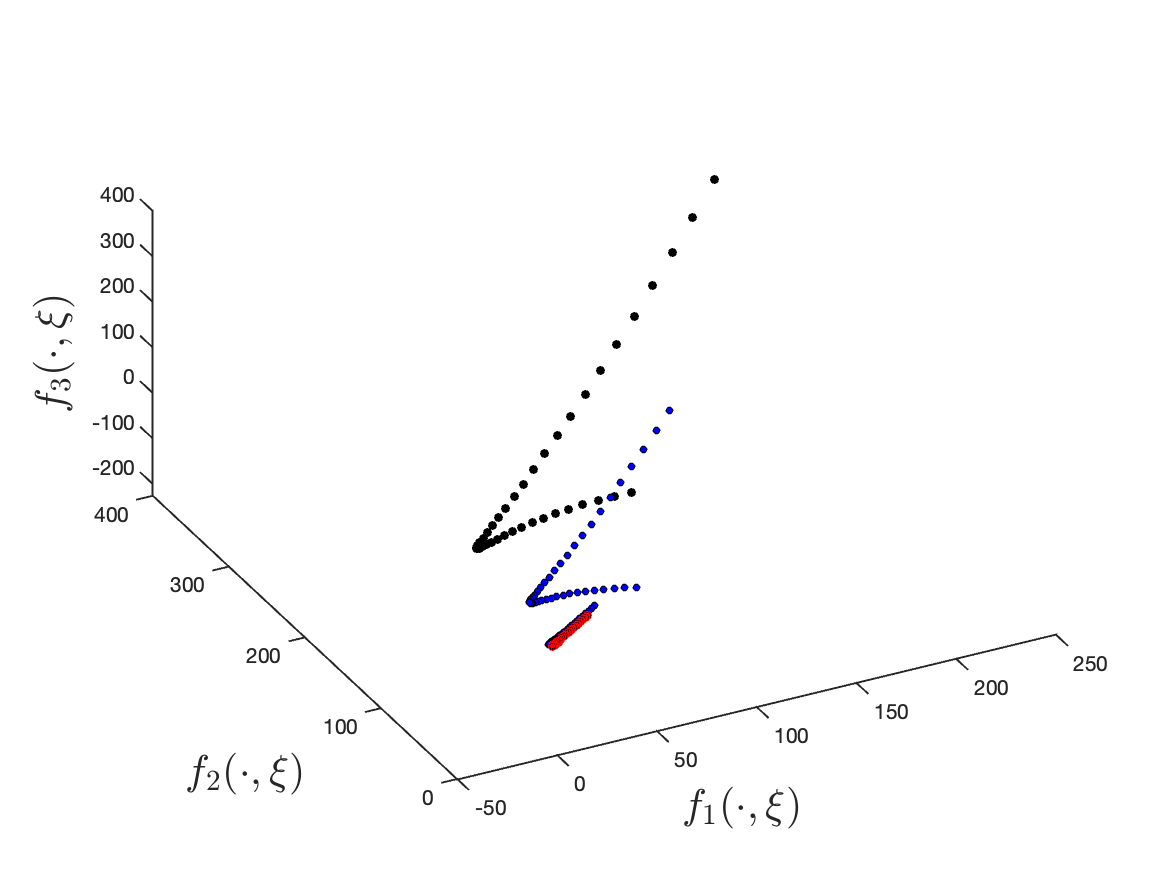}}
\hspace{0.05cm}
\caption{Robust weakly efficient points generated by Algorithm \ref{QNM_algoritm} for a few test problems from Table \ref{Table5}}\label{iterations_on_test_problems}
\end{figure}
\end{landscape}

\section{Conclusion}\label{section 6}
In this paper, we have introduced a quasi-Newton method to determine weakly robust efficient solutions for UVOPs with an uncertainty set of finite cardinality. We adopted a set-valued optimization perspective to reformulate the problem as a deterministic one. The deterministic set optimization problem was structured so that its efficient solutions, under the upper set less relation, correspond to the robust efficient solutions of the original UVOP. Utilizing the concept of the partition set, we developed a class of VOPs, which yields efficient solutions to the set optimization problem, thereby identifying robust efficient solutions for the UVOP. In the process of finding a weakly minimal solution of the UVOP, we have approximated the Hessian matrices corresponding to the given 
functions with the help of BFGS methods for vector optimization. To generate a sequence of $\beta_k$ from the partition set $P_{x_k}$ of the current iteration at $x_k$, we evaluated quasi-Newton direction $p_k$ (Step 2) with the help of concepts in \cite{27}. The process of generating iterates by Algorithm \ref{QNM_algoritm} continued until the stopping condition (Step 4) was met, and if not, then we find a step length $\tau_k$ in (Step 5) to progress to the next iterate. It has been established that if a weakly robust efficient point is a regular point, the method’s sequence exhibits global convergence (Theorem \ref{4.4}) and has a local superlinear convergence rate (see Theorem \ref{quadratic}).

In future research, we aim to explore the following directions:
\begin{itemize}
\item Extending these results and methods with respect to the other set less relations and through different scalarization functions. 
\item The goal is to create techniques for solving UMOP with uncertainty sets that are either infinite or continuous in nature.
\item Deriving results without requiring regularity assumption on the accumulation points of the sequence generated by the proposed algorithm. 
\end{itemize}

\vskip 6mm
\noindent{\bf Acknowledgments}

Debdas Ghosh acknowledges the research grant CRG (CRG/2022/001347) from SERB, India to carry out this research work. Krishna Gupta expresses gratitude to a research fellowship from IIT (BHU), Varanasi.


\begin{thebibliography}{99}














\bibitem{1}W. Zhou, F. Wang, A PRP--based residual method for large-scale monotone nonlinear equations, Appl. Math Comput.  261 (2015), 1--7.




\bibitem{2}A. Ben-Tal, A. Nemirovski, L.E. Ghaoui, Robust Optimization, Princeton University Press, Princeton, 2009.

\bibitem{3}M. Ehrgott, J. Ide, A. Sch{\"o}bel, Minmax robustness for multiobjective optimization problems, Eur. J. Oper. Res.  239 (2014), 17--31.

\bibitem{4}D. Kuroiwa, G.M. Lee, On robust multiobjective optimization, Vietnam J. Math. 40 (2012), 305--317.

\bibitem{5}M. Ehrgott, Multicriteria Optimization, Springer, Berlin, 2005.

\bibitem{6}J. Ide, A. Sch{\"o}bel, Robustness for uncertain multiobjective optimization: a survey and analysis of different concepts, Spectr. 38 (2016), 235--271.

\bibitem{7}J. Ide, E. K{\"o}bis, Concepts of efficiency for uncertain multiobjective optimization problems based on set order relations, Math. Oper. Res. 80 (2014), 99--127.

 \bibitem{8}G. Eichfelder, J. Jahn, Vector optimization problems and their solution concepts, In: Recent Developments in Vector Optimization, vol. 1, pp. 1–27, Springer, Berlin, 2011.

\bibitem{9} E. K{\"o}bis, On robust optimization -- a unified approach to robustness using a nonlinear scalarizing functional and relations to set optimization, Universit{\"a}ts-und Landesbibliothek Sachsen-Anhalt, 2014.

\bibitem{10}H.Z. Wei, C.R. Chen, S.J. Li, A unified characterization of multiobjective robustness via separation, J. Optim. Theory Appl. 179 (2018), 86--102.


\bibitem{11}F. Wang, S. Liu, Y. Chai, Robust counterparts and robust efficient solutions in vector optimization under uncertainty, Oper. Res. Lett. 43 (2015), 293--298.

\bibitem{12}C. Liu, Z. Gong, K.L. Teo, J. Sun, L. Caccetta, Robust multiobjective optimal switching control arising in 1, 3-propanediol microbial fed-batch process, Nonlinear Anal.: Hybrid Syst. 25 (2017), 1--20.

\bibitem{13}D. Ghosh, N. Kishor, Generalized ordered weighted aggregation robustness: a new robust counterpart to solve uncertain multiobjective optimization problems, Eng. Optim. (2024), 1--32. \\ 
https://doi.org/10.1080/0305215X.2024.2418341



\bibitem{14}N. Kishor, D. Ghosh, X. Zhao, Generalized ordered weighted aggregation robustness to solve uncertain single objective optimization problems, arXiv preprint arXiv:2410.03222 (2024). 


\bibitem{15}S. Kumar, M.A.T. Ansary, N.K. Mahato, D. Ghosh, Y. Shehu, Newton's method for uncertain multiobjective optimization problems under finite uncertainty sets, J. Nonlinear Var. Anal. 7 (2023), 785--809.

\bibitem{16}S. Kumar, M.A.T. Ansary, N.K. Mahato, D. Ghosh, Steepest descent method for uncertain multiobjective optimization problems under finite uncertainty set, Appl. Anal. (2024), 1--22. \\ https://doi.org/10.1080/00036811.2024.2426219

\bibitem{17}D. Ghosh, N. Kishor, X. Zhao, A Newton method for uncertain multiobjective optimization problems with finite uncertainty set, J. Nonlinear Var. Anal. 9 (2024), 81--110.


\bibitem{36}J. Ide, E. Köbis, D. Kuroiwa, A. Schöbel, C. Tammer, The relationship between multi-objective robustness concepts and set-valued optimization, Fixed Point Theory Appl. 83 (2014), 1--20.

\bibitem{37}D. Kuroiwa, Existence theorems of set optimization with set-valued maps, J Inform Optim Sci. 24 (2003), 73--84.


\bibitem{38}D. Kuroiwa, Natural Criteria of Set-Valued Optimization, Manuscript Shimane University, Japan, 1068 (1999), 164--170. 


\bibitem{39}D. Kuroiwa, Some duality theorems of set-valued optimization with natural criteria, In: T. Tanaka, (ed.) 
Proceedings of the International Conference on  Nonlinear Analysis and Convex Analysis, pp. 221–228, World Scientific, Singapore, 1999.


\bibitem{40}A. Lovison, Singular continuation: Generating piecewise linear approximations to Pareto sets via global analysis, SIAM J. Optim. 21 (2011), 463--490. 


\bibitem{18}G. Bouza, E. Quintana, C. Tammer, A steepest descent method for set optimization problems with set-valued mappings of finite cardinality, J. Optim. Theory Appl. 190 (2021), 711--743.



\bibitem{19}J.E. Dennis Jr., J.J. Mor{\'e}, Quasi-Newton methods, motivation and theory, SIAM Rev. 19 (1977), 46--89.

\bibitem{20}A. Göpfert, H. Riahi, C. Tammer, C. Zalinescu, Variational Methods in Partially Ordered Spaces, Springer, New York, 2003.


\bibitem{21}J. Jahn, T.X.D. Ha, New order relations in set optimization, J. Optim. Theory Appl. 148 (2011), 209--236.


\bibitem{22}L.M.G. Drummond, F.M.P. Ruppa, B.F. Svaiter, A quadratically convergent Newton method for vector optimization, Optim. 63 (2014), 661--677.

\bibitem{23}A.A. Khan, C. Tammer, C. Zălinescu, Set-valued Optimization, Springer, Berlin, 2014.

\bibitem{24}S.J. Wright, Numerical Optimization, Springer, New York, 2006.


\bibitem{25}J. Fliege, L.M.G. Drummond, B.F. Svaiter, Newton's method for multiobjective optimization, SIAM J. Optim. 20 (2009), 602--626. 

\bibitem{26}D. Ghosh, A. Anshika, J.-C. Yao, X. Zhao, Quasi-Newton method for set optimization problems with set-valued mapping given by finitely many vector-valued functions, arXiv preprint, arXiv:2501.04711 (2024). 

\bibitem{27}Ž. Povalej, Quasi-Newton’s method for multiobjective optimization, J. Comput. Appl. Math. 255 (2014), 765--777.

\bibitem{32}S. Huband, P. Hingston, L. Barone, 
L. While, A review multiobjective test problem and a scalable test problem toolkit, IEEE Trans. Evol. Comput. 10 (2006), 477--506.


\bibitem{33} Y. Jin, M. Olhofer, B. Sendhoff, Dynamic weighted aggregation for evolutionary multiobjective optimization: Why does it work and how?, Proceedings of the Genetic and Evolutionary
Computation Conference, pp. 1042--1049, 2001.

\bibitem{34} J. Fliege, L.G. Drummond, B.F. Svaiter, Newton’s method for multiobjective optimization, SIAM J. Optim. 20 (2009), 602--626.


\bibitem{35} E. Zitzler, K. Deb, L. Thiele, Comparison of multiobjective evolutionary algorithms: Empirical results, Evol. Comput. 8 (2000), 173--195.




\bibitem{28}C.G. Broyden, A new double--rank minimisation algorithm, Preliminary report, Am. Math. Soc. Notices 16 (1969), 670.

\bibitem{29}R. Fletcher, A new approach to variable metric algorithms, Comput. J. 13 (1970), 317--322.


\bibitem{30}D. Goldfarb, A family of variable--metric methods derived by variational means, Math. Comput. 24 (1970), 23--26.


\bibitem{31}D.F. Shanno, Conditioning of quasi-Newton methods for function minimization, Math. Comput. 24 (1970), 647--656.







\end{thebibliography}

\end{document}